\documentclass[a4paper, 11pt]{amsart}
\usepackage{amsmath, amssymb, amsthm, mathrsfs}
\usepackage[mathscr]{eucal}
\usepackage[all]{xy}

\theoremstyle{plain} 
 \newtheorem{thm}{Theorem}[section]
 \newtheorem{lem}[thm]{Lemma}
 \newtheorem{cor}[thm]{Corollary}
 \newtheorem{prop}[thm]{Proposition}
 
\theoremstyle{definition}
  \newtheorem{defn}[thm]{Definition}
  
  \newtheorem{notation}[thm]{Notation}

\theoremstyle{remark}
  \newtheorem{rem}[thm]{Remark}

\newcommand{\C}{\mathbb{C}}
\newcommand{\cal}{\mathcal}
\newcommand{\N}{\mathbb{N}}
\newcommand{\Z}{\mathbb{Z}}
\newcommand{\calr}{\mathcal{R}}

\newcommand{\calg}{\mathcal{G}}
\newcommand{\calh}{\mathcal{H}}
\newcommand{\cals}{\mathcal{S}}
\newcommand{\cali}{\mathcal{I}}
\newcommand{\sh}{\mathscr{H}}

\newcommand{\scc}{\mathscr{C}}
\newcommand{\sd}{\mathscr{D}}
\newcommand{\se}{\mathscr{E}}
\newcommand{\scb}{\mathscr{B}}

\newcommand{\calq}{\mathcal{Q}}
\newcommand{\sff}{\mathsf{F}}
\newcommand{\ci}[2]{\cite[#1]{#2}}
\renewcommand{\c}{\curvearrowright}

  \makeatletter
  \@addtoreset{equation}{section}
  \makeatother

\begin{document}

\title[Splitting in orbit equivalence]{Splitting in orbit equivalence, treeable groups, and the Haagerup property}
\author{Yoshikata Kida}
\address{Department of Mathematics, Kyoto University, 606-8502 Kyoto, Japan}
\email{kida@math.kyoto-u.ac.jp}
\date{January 29, 2015}
\thanks{This work was supported by JSPS Grant-in-Aid for Young Scientists (B), No.25800063}

\begin{abstract}
Let $G$ be a discrete countable group and $C$ its central subgroup with $G/C$ treeable.
We show that for any treeable action of $G/C$ on a standard probability space $X$, the groupoid $G\ltimes X$ is isomorphic to the direct product of $C$ and
$(G/C)\ltimes X$, through cohomology of groupoids.
We apply this to show that any group in the minimal class of groups containing treeable groups and closed under taking direct products, commensurable groups and central extensions has the Haagerup property.
\end{abstract}

\maketitle


\section{Introduction}\label{sec-intro}

It is an elementary fact that any group extension of a free group splits.
In the framework of measured groupoids, Series \cite{series} observed that any extension of the equivalence relation associated with a free action of a free group splits analogously.
The same splitting therefore holds for any extension of an ergodic treeable equivalence relation of type ${\rm II}_1$, thanks to Gaboriau \cite{gab-cost} and Hjorth \cite{hj}.
A discrete measured equivalence relation is called treeable, roughly speaking, if it has a measurable bundle structure whose fiber is a simplicial tree.
Treeability was introduced by Adams \cite{ad}, and is widely regarded as an analogue of freeness of groups.
In the first half of the paper, we show that the same splitting holds for any central extension of a treeable equivalence relation of type ${\rm II}_1$ which is not necessarily ergodic.
In the second half of the paper, we apply this to show the Haagerup property of certain central extensions of discrete countable groups.

In this introduction, we avoid to precisely define a central extension of an equivalence relation.
We instead state the splitting theorem for the case where the central extension arises from a p.m.p.\ action of a discrete countable group $G$ such that a central subgroup of $G$ acts trivially. 
We mean by a {\it p.m.p.}\ action of a discrete countable group $G$ a measurable action of $G$ on a standard probability space preserving the measure, where ``p.m.p." stands for ``probability-measure-preserving".

\begin{thm}\label{thm-split-groupoid}
Let $1\to C\to G\to \Gamma \to 1$ be an exact sequence of discrete countable groups such that $C$ is central in $G$.
We denote by $q\colon G\to \Gamma$ the quotient map.
Let $\Gamma \c (X, \mu)$ be a free and p.m.p.\ action such that the associated equivalence relation is treeable.
Let $G$ act on $(X, \mu)$ through $q$.
Then there exists an isomorphism of discrete measured groupoids,
\[I\colon C\times (\Gamma \ltimes (X, \mu))\to G\ltimes (X, \mu),\]
such that for any $c\in C$ and a.e.\ $x\in X$, we have $I(c, (e, x))=(c, x)$, where $e$ is the neutral element of $\Gamma$; and for any $g\in G$ and a.e.\ $x\in X$, the coordinate of $I^{-1}(g, x)$ in $\Gamma \ltimes (X, \mu)$ is $(q(g), x)$.
\end{thm}

We refer to Theorem \ref{thm-section} for a general version of Theorem \ref{thm-split-groupoid}.
Using this splitting, we show that there exists an ergodic, free and p.m.p.\ action of $G$ which is orbit equivalent to the direct product of such actions of $C$ and of $\Gamma$ (see Theorem \ref{thm-split}).

Theorem \ref{thm-split-groupoid} for the case where the action $\Gamma \c (X, \mu)$ is orbit equivalent to a free p.m.p.\ action of a free group $F$ is derived from Series' observation based on cohomology of discrete measured groupoids (\cite[\S 4, (C)]{series}).
The proof is outlined as follows:
Let $\calr$ be the discrete measured equivalence relation associated with the action $\Gamma \c (X, \mu)$.
Analogously to the theory of group extensions, there is one-to-one correspondence between a central extension of $\calr$ with a central subgroup $C$ and an element of the cohomology group $H^2(\calr, C)$, such that the splitting of the extension corresponds to $0$.
The group $H^2(\calr, C)$ is isomorphic to the second cohomology group of $F$ with coefficient in some module, thanks to Westman \cite{west2}.
The latter vanishes because $F$ is a free group.
The splitting then follows.

\begin{rem}
For any p.m.p.\ treeable equivalence relation $\calr$ and any abelian Polish group $C$, regarding $C$ as an $\calr$-module on which $\calr$ acts trivially, we indeed show that $H^n(\calr, C)=0$ for any integer $n$ with $n\geq 2$ (Corollary \ref{cor-treeable}).

Treeability was extended for discrete measured groupoids by Anantharaman-Delaroche \cite{ana-t}, \cite{ana-h} and Ueda \cite{ueda}, where several features of treeable equivalence relations are extended.
It is natural to ask whether the above vanishing of cohomology also holds for any treeable groupoid.
This question remains unsolved because we do not know whether any ergodic treeable groupoid arises from an action of a free group, up to stable isomorphism.
\end{rem}

In \cite[\S 7.3.3]{ccjjv}, the following question on the Haagerup property (HAP) of groups (\cite{h}) is asked:
For a central extension $1\to C\to G\to \Gamma \to 1$ of locally compact groups, is it true that $G$ has the HAP if and only if $\Gamma$ has the HAP?
We give new examples of a discrete countable group $G$ and its central subgroup $C$ such that $G$ and $G/C$ have the HAP.
They are obtained as an application of the splitting theorem.

A discrete countable group is called {\it treeable} if it has an ergodic, free and p.m.p.\ action such that the associated equivalence relation is treeable.
We refer to \cite{al}, \cite{btw} and \cite{gab-ex} for examples of treeable groups.
If a discrete countable group $G$ has a central subgroup $C$ with $G/C$ treeable, then $G$ has the HAP because $G$ is measure equivalent to the direct product $C\times (G/C)$ by Theorem \ref{thm-split-groupoid} (or Theorem \ref{thm-split}) and the HAP is invariant under measure equivalence (\cite[Proposition 1.3]{jol-h-vn}, \cite[Remark in p.172]{jol-h} and \cite[Proposition 3.1]{popa}).
Using Theorem \ref{thm-split-groupoid} and its general version, we further obtain a much broader class of central extensions having the HAP.

\begin{notation}
Let $\scc$ denote the smallest subclass of the class of discrete countable groups that satisfies the following conditions {\rm (1)--(4)}:
\begin{enumerate}
\item[(1)] Any treeable group belongs to $\scc$.
\item[(2)] The direct product of two groups in $\scc$ belongs to $\scc$.
\item[(3)] For a discrete countable group $G$ and its finite index subgroup $H$, we have $G\in \scc$ if and only if $H\in \scc$.
\item[(4)] Any central extension of a group in $\scc$ belongs to $\scc$.
\end{enumerate}
\end{notation}

\begin{thm}\label{thm-h}
Any group in $\scc$ has the HAP.
More generally, for any discrete countable group $G$ and its central subgroup $C$ with $G/C\in \scc$, the pair $(G, C)$ has the gHAP.
\end{thm}

The generalized Haagerup property (gHAP) is defined for the pair $(G, C)$ of a locally compact group $G$ and its closed central subgroup $C$ (\cite[Definition 4.2.1]{ccjjv}).
This is used to discuss the HAP of a connected Lie group and its central extension.
It is shown that for any two pairs $(A, C)$, $(B, C)$ of discrete countable groups with the gHAP, the amalgamated free product $A\ast_CB$ has the gHAP with respect to $C$, and thus has the HAP (\cite[Proposition 6.2.3 (2)]{ccjjv}).
We should mention the following fundamental facts on the gHAP:
Let $G$ be a locally compact group and $C$ its closed central subgroup.
\begin{itemize}
\item If $(G, C)$ has the gHAP, then $G/C$ has the HAP.
\item The pair $(G, \{ e\})$ has the gHAP if and only if $G$ has the HAP.
\item If $G$ is second countable and $(G, C)$ has the gHAP, then $G$ has the HAP.
\item If $G$ is amenable, then $(G, C)$ has the gHAP.
\end{itemize}
We refer to Lemma 4.2.8, Lemma 4.2.10, Lemma 4.2.11 and Proposition 4.2.12 in \cite{ccjjv} for these assertions, respectively.
It follows that if $A$ and $B$ are discrete countable amenable groups with a common central subgroup $C$, then $A\ast_CB$ has the HAP (\cite[Corollary 6.2.5]{ccjjv}).
Theorem \ref{thm-h} covers this result.

Our initial motivation of this work is to find groups having an ergodic, free, p.m.p.\ and stable action.
We say that an ergodic p.m.p.\ action is {\it stable} if the associated groupoid is isomorphic to its direct product with the ergodic hyperfinite equivalence relation of type ${\rm II}_1$.
The author shows that for a discrete countable group $G$ with a central subgroup $C$, if $(G, C)$ does not have property (T), then $G$ has such a stable action (\cite[Theorem 1.1]{kida-srt}).
By Theorem \ref{thm-h}, we therefore obtain the following:

\begin{cor}
Any group in $\scc$ with the infinite center has an ergodic, free, p.m.p.\ and stable action, and is therefore measure equivalent to its direct product with $\Z$.
\end{cor}


The paper is organized as follows:
In Section \ref{sec-groupoid}, we prepare terminology and notation on discrete measured groupoids.
We review cohomology of discrete measured groupoids and its basic properties, due to Westman \cite{west1}, \cite{west2}.
The relationship with cohomology of their ergodic components are also discussed.
In Section \ref{sec-split}, we prove the splitting theorem \ref{thm-section} for a central extension of a treeable equivalence relation.

In Section \ref{sec-hap}, we discuss the HAP and gHAP of discrete measured groupoids.
The former was originally introduced by Anantharaman-Delaroche \cite{ana-h}.
For a technical requirement, changing her definition slightly, we introduce another property and call it the HAP in this paper.
We review a Hilbert bundle and a representation of a discrete measured groupoid on it, following \cite{ana-t}, \cite{ramsay} and \cite{wil}.
We discuss induced representations, the $c_0$-property of a representation, and the weak containment of the trivial representation.
One of the aim of Section \ref{sec-hap} is to show that our HAP of a discrete measured groupoid is preserved under taking an extension whose quotient is an abelian group (Proposition \ref{prop-hap-ext}).
In Section \ref{sec-c}, applying it, we prove Theorem \ref{thm-h}.

In Appendix \ref{sec-app}, we compare the HAPs of Jolissaint, Anantharaman-Delaroche, and ours.
Jolissaint \cite{jol-h} introduced the HAP for p.m.p.\ discrete measured equivalence relations.
His definition can naturally be extended for discrete measured groupoids.
We show that this extension of the HAP of Jolissaint, the HAP of Anantharaman-Delaroche, and our HAP are all equivalent.

Throughout the paper, for a set $S$, we denote by $|S|$ its cardinality.
Let $\N$ denote the set of positive integers.
Unless otherwise mentioned, we mean by a discrete group a discrete and countable group.

\medskip

\noindent {\it Acknowledgements.} We thank Yoshimichi Ueda for informing us of Series' work \cite{series}.


\section{Discrete measured groupoids and their cohomology}\label{sec-groupoid}

We mean by a {\it standard probability space} $(X, \mu)$ a standard Borel space $X$ endowed with a probability measure $\mu$.
We refer to \cite{arv} and \cite{kec-set} for its fundamentals.
Throughout this section, we fix a standard probability space $(X, \mu)$.

\subsection{Terminology}

Recall that a subset $A$ of $X$ is called {\it $\mu$-measurable} if there exist Borel subsets $B_1$, $B_2$ of $X$ such that $B_1\subset A\subset B_2$ and $\mu(B_2\setminus B_1)=0$.
When $\mu$ is understood from the context, let us simply call it a {\it measurable} subset of $X$.
We also call $\mu$-measurable functions, maps, etc.\ measurable ones.
All relations among measurable sets and maps that appear in this paper are understood to hold up to sets of measure zero, unless otherwise mentioned.

We follow the terminology employed in \cite{ana-t} and \cite{ana-h} on discrete measured groupoids.
Let $\calg$ be a groupoid on $X$ with $r, s\colon \calg \to X$ the range and source maps, respectively.
We set
\[\calg^{(2)}=\{\, (g, h)\in \calg \times \calg \mid s(g)=r(h)\, \}.\]
For $x, y\in X$, we set $\calg^y=r^{-1}(y)$ and $\calg_x=s^{-1}(x)$.
We say that $\calg$ is {\it Borel} if it is endowed with a standard Borel structure such that the range, source, inverse and product maps are Borel, where $\calg^{(2)}$ is endowed with the standard Borel structure induced by $\calg \times \calg$.
We say that $\calg$ is {\it discrete} if for any $y\in X$, the set $\calg^y$ is countable.

Suppose that $\calg$ is Borel and discrete.
Let $\tilde{\mu}$ be the $\sigma$-finite measure on $\calg$ defined so that for any Borel subset $E$ of $\calg$, we have
\[\tilde{\mu}(E)=\int_X|E\cap \calg_x|\, d\mu(x).\]
Let $\tilde{\mu}^{-1}$ be the image of $\tilde{\mu}$ under the inverse map of $\calg$.
We say that $\mu$ is {\it quasi-invariant} under $\calg$ if $\tilde{\mu}$ and $\tilde{\mu}^{-1}$ are equivalent.
We say that $\mu$ is {\it invariant} under $\calg$ if $\tilde{\mu}=\tilde{\mu}^{-1}$.

We say that $\calg$ is a {\it discrete measured groupoid} on $(X, \mu)$ if $\calg$ is Borel and discrete and the measure $\mu$ is quasi-invariant under $\calg$.
If $\mu$ is invariant under $\calg$, then $\calg$ is called {\it p.m.p.}
Unless otherwise mentioned, we mean by a {\it subgroupoid} of $\calg$ a measurable subgroupoid of $\calg$ whose unit space is equal to $X$.

For a measurable subset $A$ of $X$, we define the {\it saturation} of $A$ by $\calg$ as
\[\calg A=\{ \, r(g)\in X\mid s(g)\in A\, \},\]
which is a measurable subset of $X$. 
We say that $\calg$ is {\it ergodic} if for any measurable subset $A$ of $X$, we have $\mu(\calg A)=0$ or $1$.
For a measurable subset $A$ of $X$ with positive measure, we define the {\it restriction} of $\calg$ to $A$ as
\[\calg|_A=\{ \, g\in \calg \mid r(g), s(g)\in A\, \},\]
which is a discrete measured groupoid on $(A, \mu|_A)$.

Let $\calh$ be a discrete measured groupoid on a standard probability space $(Y, \nu)$.
We mean by a {\it homomorphism} from $\calg$ into $\calh$ a measurable map $\alpha \colon \calg \to \calh$ such that the measure $\alpha_*\tilde{\mu}$ on $\calh$ is absolutely continuous with respect to $\tilde{\nu}$, and for a.e.\ $(g_1, g_2)\in \calg^{(2)}$, we have $(\alpha(g_1), \alpha(g_2))\in \calh^{(2)}$ and $\alpha(g_1g_2)=\alpha(g_1)\alpha(g_2)$.
If there exists a homomorphism $\beta \colon \calh \to \calg$ such that $\beta \circ \alpha$ is the identity on $\calg$ and $\alpha \circ \beta$ is the identity on $\calh$, then $\alpha$ is called an {\it isomorphism}.
We say that $\calg$ and $\calh$ are {\it isomorphic} if there exists an isomorphism between them.
We say that $\calg$ and $\calh$ are {\it stably isomorphic} if there exist measurable subsets $A\subset X$ and $B\subset Y$ such that $\calg A=X$, $\calh B=Y$, and $\calg|_A$ and $\calh|_B$ are isomorphic.

Let $G$ be a discrete group.
Suppose that we have a {\it non-singular} action $G\c (X, \mu)$, i.e., a Borel action of $G$ on $X$ preserving the class of $\mu$.
The product set $G\times X$ is then endowed with a structure of a discrete measured groupoid on $(X, \mu)$ as follows:
The range and source maps are defined by $r(\gamma, x)=\gamma x$ and $s(\gamma, x)=x$ for $\gamma \in G$ and $x\in X$.
The product is defined by $(\gamma, \delta x)(\delta, x)=(\gamma \delta, x)$ for $\gamma, \delta \in G$ and $x\in X$.
The unit at $x\in X$ is $(e, x)$, where $e$ is the neutral element of $G$.
The inverse of $(\gamma, x)\in G\times X$ is $(\gamma^{-1}, \gamma x)$.
The measure $\mu$ is quasi-invariant because the action $G\c (X, \mu)$ is non-singular.
This discrete measured groupoid is denoted by $G\ltimes (X, \mu)$, and is called the groupoid {\it associated} with the action $G\c (X, \mu)$.
We note that the groupoid $G\ltimes (X, \mu)$ is p.m.p.\ if and only if the action $G\c (X, \mu)$ is p.m.p.

For a non-singular action $G\c (X, \mu)$, we define the equivalence relation {\it associated} with the action as the discrete measured groupoid on $(X, \mu)$,
\[\calr =\{ \, (\gamma x, x)\in X\times X\mid \gamma \in G,\ x\in X\, \},\]
where the range and source maps are defined by $r(\gamma x, x)=\gamma x$ and $s(\gamma x, x)=x$ for $\gamma \in G$ and $x\in X$; the product is defined by $(x, y)(y, z)=(x, z)$ for $(x, y), (y, z)\in \calr$; the unit at $x\in X$ is $(x, x)$; and the inverse of $(x, y)\in \calr$ is $(y, x)$.
We have the quotient homomorphism $q \colon G\ltimes (X, \mu)\to \calr$ defined by $q(\gamma, x)=(\gamma x, x)$ for $\gamma \in G$ and $x\in X$.


\subsection{Cohomology}\label{subsec-coh}

Throughout this subsection, we fix a discrete measured groupoid $\calg$ on $(X, \mu)$, and fix an abelian Polish group $C$.
We review cohomology of $\calg$ with coefficient in the $\calg$-module $C$ on which $\calg$ acts trivially, due to Westman \cite{west1}, \cite{west2}.
We refer to Feldman-Moore \cite{fm1} and Series \cite{series} for cohomology of discrete measured equivalence relations with more general coefficients.

Let $r, s\colon \calg \to X$ be the range and source maps of $\calg$, respectively.
We set $\calg^{(0)}=X$, and for a positive integer $n$, we set
\[\calg^{(n)}=\{ \, (g_1,\ldots, g_n)\in \calg^n\mid s(g_i)=r(g_{i+1}) \ \textrm{for\ any}\ i=1, 2,\ldots, n-1\, \}.\]
We define a map $s_n\colon \calg^{(n)}\to X$ by $s_n(g_1,\ldots, g_n)=s(g_n)$ for $(g_1,\ldots, g_n)\in \calg^{(n)}$.
For any $x\in X$, the fiber $s_n^{-1}(x)$ at $x$ is countable.
We define a measure $\mu_n$ on $\calg^{(n)}$ by
\[\mu_n(E)=\int_X|E\cap s_n^{-1}(x)|\, d\mu(x)\]
for a measurable subset $E$ of $\calg^{(n)}$.
Note that we have $\mu^{(1)}=\tilde{\mu}$ on $\calg^{(1)}=\calg$.

The product of two elements of the abelian group $C$ are denoted by addition.
For any non-negative integer $n$, we define $C^n(\calg, C)$ as the module of measurable maps from $\calg^{(n)}$ into $C$, where two maps equal almost everywhere are identified, and the module structure is endowed by pointwise addition.
We define a homomorphism
\[\delta^n\colon C^n(\calg, C)\to C^{n+1}(\calg, C)\]
as follows:
For $\phi \in C^0(\calg, C)$, we define $\delta^0\phi \in C^1(\calg, C)$ by
\[(\delta^0\phi)(g)=\phi(s(g))-\phi(r(g))\]
for $g\in \calg$.
If $n\geq 1$, then for $\phi \in C^n(\calg, C)$, we define $\delta^n\phi \in C^{n+1}(\calg, C)$ by
\begin{align*}
(\delta^n\phi)(g_0, g_1,\ldots, g_n)&=\phi(g_1,\ldots, g_n)+\sum_{k=1}^n(-1)^k\phi(g_0,\ldots, g_{k-2}, g_{k-1}g_k, g_{k+1},\ldots, g_n)\\
& \qquad +(-1)^{n+1}\phi(g_0,\ldots, g_{n-1})
\end{align*}
for $(g_0, g_1, \ldots, g_n)\in \calg^{(n+1)}$.
By direct computation, we have $\delta^n\circ \delta^{n-1}=0$ for any positive integer $n$.
It follows that $(C^\bullet(\calg, C), \delta^\bullet)$ is a cochain complex.

Let $n$ be a non-negative integer.
We set $Z^n(\calg, C)=\ker \delta^n$.
We set $B^0(\calg, C)=0$, and set $B^n(\calg, C)=\delta^{n-1}(C^{n-1}(\calg, C))$ if $n\geq 1$.
We define $H^n(\calg, C)$ as the quotient module $Z^n(\calg, C)/B^n(\calg, C)$, and call it the $n$-th {\it cohomology group} of $\calg$ with coefficient in $C$.

Let $\calh$ be a discrete measured groupoid on a standard probability space $(Y, \nu)$.
Let $n$ be a non-negative integer.
Any homomorphism $\alpha \colon \calg \to \calh$ induces a homomorphism from $C^n(\calh, C)$ into $C^n(\calg, C)$ compatible with the coboundary map $\delta^n$.
We therefore have the induced homomorphism $\alpha^*\colon H^n(\calh, C)\to H^n(\calg, C)$.

Let $\alpha, \beta \colon \calg \to \calh$ be homomorphisms.
Let $\alpha_0, \beta_0\colon X\to Y$ denote the maps induced by $\alpha$ and $\beta$, respectively.
We say that $\alpha$ and $\beta$ are {\it similar} if there exists a measurable map $\varphi \colon X\to \calh$ such that for a.e.\ $x\in X$, the range and source of $\varphi(x)$ are $\beta_0(x)$ and $\alpha_0(x)$, respectively, and for a.e.\ $g\in \calg$, we have $\varphi(r(g))\alpha(g)=\beta(g)\varphi(s(g))$.
The following lemma is proved by constructing a cochain homotopy (see the proof of \cite[Theorem 2.3]{west1}).

\begin{lem}[\ci{Theorems 2.3 and 3.55}{west1}]\label{lem-similar}
Let $\alpha, \beta \colon \calg \to \calh$ be similar homomorphisms.
Then for any non-negative integer $n$, the induced homomorphisms $\alpha^*, \beta^*\colon H^n(\calh, C)\to H^n(\calg, C)$ are equal.
\end{lem}

As an application of Lemma \ref{lem-similar}, we obtain the following:

\begin{lem}\label{lem-rest}
Let $A$ be a measurable subset of $X$ with $\calg A=X$.
Let $i\colon \calg|_A\to \calg$ be the homomorphism defined as the inclusion.
Then for any non-negative integer $n$, the induced homomorphism $i^*\colon H^n(\calg, C)\to H^n(\calg|_A, C)$ is an isomorphism.
\end{lem}

\begin{proof}
By the assumption $\calg A=X$, there exists a measurable map $f\colon X\to \calg$ such that for any $x\in X$, we have $s(f(x))=x$ and $r(f(x))\in A$; and for any $x\in A$, the element $f(x)$ is the unit of $\calg$ at $x$.
We define a homomorphism $j\colon \calg \to \calg|_A$ by $j(g)=f(r(g))gf(s(g))^{-1}$ for $g\in \calg$.
The composition $j\circ i$ is the identity on $\calg|_A$, and the composition $i\circ j$ is the homomorphism from $\calg$ into itself similar to the identity on $\calg$.
The lemma follows from Lemma \ref{lem-similar}.
\end{proof}

Let $G\c (X, \mu)$ be a non-singular action of a discrete group $G$.
Let $\sff (X, C)$ denote the module of measurable maps from $X$ into $C$, where two maps equal almost everywhere are identified, and the module structure is endowed by pointwise addition.
We endow $\sff(X, C)$ with a $G$-module structure by $(\gamma f)(x)=f(\gamma^{-1}x)$ for $\gamma \in G$, $f\in \sff(X, C)$ and $x\in X$.
Let $(C^\bullet(G, \sff(X, C)), \delta^\bullet)$ denote the usual cochain complex of the group $G$ with coefficient in the $G$-module $\sff(X, C)$.
For any non-negative integer $n$, the natural isomorphism
\[\pi_n\colon C^n(G\ltimes (X, \mu), C)\to C^n(G, \sff(X, C))\]
is defined by
\[(\pi_n\phi)(\gamma_1,\ldots, \gamma_n)(x)=\phi((\gamma_1, \gamma_1^{-1}x), (\gamma_2, \gamma_2^{-1}\gamma_1^{-1}x),\ldots, (\gamma_n, \gamma_n^{-1}\gamma_{n-1}^{-1}\cdots \gamma_1^{-1}x))\]
for $\phi \in C^n(G\ltimes (X, \mu), C)$, $\gamma_1,\ldots, \gamma_n\in G$ and $x\in X$.
The map $\pi_n$ is compatible with the coboundary map.
We therefore obtain the following theorem that relates cohomology of $G\ltimes (X, \mu)$ to that of $G$.

\begin{thm}[\ci{Theorem 1.0}{west2}]\label{thm-coh}
Let $G\c (X, \mu)$ be a non-singular action of a discrete group $G$.
For any non-negative integer $n$, the homomorphism $\pi_n$ induces the isomorphism from $H^n(G\ltimes (X, \mu), C)$ onto $H^n(G, \sff(X, C))$.
\end{thm}

The following is a generalization of \cite[Corollary 1.1]{west2}.

\begin{cor}\label{cor-tree}
Let $\calr$ be a discrete measured equivalence relation which is ergodic, p.m.p.\ and treeable.
Let $C$ be an abelian Polish group, and regard it as an $\calr$-module on which $\calr$ acts trivially.
Then $H^n(\calr, C)=0$ for any integer $n$ with $n\geq 2$.
\end{cor}

\begin{proof}
Thanks to Gaboriau \cite[Proposition II.6]{gab-cost} and Hjorth \cite[Corollary 1.2]{hj}, we have a free group $F$ and its ergodic, free and p.m.p.\ action $F\c (Y, \nu)$ such that $\calr$ is stably isomorphic to the groupoid $F\ltimes (Y, \nu)$.
For any non-negative integer $n$, the group $H^n(\calr, C)$ is isomorphic to $H^n(F\ltimes (Y, \nu), C)$ by Lemma \ref{lem-rest}, and to $H^n(F, \sff(Y, C))$ by Theorem \ref{thm-coh}.
If $n\geq 2$, then the last cohomology group is zero because $F$ is a free group.
\end{proof}


\subsection{Restrictions to ergodic components}

Let $\calr$ be a discrete measured equivalence relation on $(X, \mu)$.
Let $C$ be an abelian Polish group, and regard it as an $\calr$-module on which $\calr$ acts trivially.
In this subsection, we show that if two cocycles on $\calr$ into $C$ are cohomologous on almost every ergodic component of $\calr$, then they are indeed cohomologous as cocycles on $\calr$.
The proof relies on Fisher-Morris-Whyte \cite{fmw}.
We do not deal with the same problem for a general groupoid here.

Let $\theta \colon (X, \mu)\to (Z, \xi)$ be the ergodic decomposition for $\calr$.
Let $\mu =\int_Z\mu_z\, d\xi(z)$ be the disintegration of $\mu$ with respect to $\theta$.
For $z\in Z$, we define $\calr_z$ as the discrete measured equivalence relation $\calr$ on $(X, \mu_z)$, which is ergodic and is also regarded as a groupoid on $\theta^{-1}(z)$.
It is known that there exists a measurable action of a discrete group $G$ on $(X, \mu)$ whose orbit equivalence relation is $\calr$ (\cite[Theorem 1]{fm1}).
The ergodic decomposition for $\calr$ is identified with that for the action of $G$.

For a non-negative integer $n$, a measurable map $\phi \colon \calr^{(n)}\to C$ and $z\in Z$, we denote by $\phi_z\colon (\calr_z)^{(n)}\to C$ the restriction of $\phi$.

\begin{prop}\label{prop-comp}
We keep the above notation.
Fix a non-negative integer $n$, and pick a measurable map $\phi \colon \calr^{(n)}\to C$.
If $\phi_z\in B^n(\calr_z, C)$ for a.e.\ $z\in Z$, then $\phi\in B^n(\calr, C)$.
\end{prop}

To prove this proposition, we prepare notation and a lemma.
Let $(S, \eta)$ be a standard probability space and $(M, d)$ a separable metric space.
As in Subsection \ref{subsec-coh}, we define $\sff(S, M)$ as the space of measurable maps from $S$ into $M$, where two maps are identified if they are equal almost everywhere.
If $(M, d)$ is complete, then $\sff(S, M)$ is a Polish space, under the topology of convergence in measure (\cite[Notation 2.4]{fmw}).

\begin{lem}\label{lem-map}
Let $(S, \eta)$ be a standard probability space, $(M, d)$ a separable metric space, and $L$ a standard Borel space.
Then the following two assertions hold:
\begin{enumerate}
\item[(i)] For any Borel map $\varphi \colon L\times S\to M$, the induced map $\widetilde{\varphi}\colon L\to \sff(S, M)$ is defined by $(\widetilde{\varphi}(x))(s)=\varphi(x, s)$ for $x\in L$ and $s\in S$, and is Borel.
\item[(ii)] Conversely, for any Borel map $\Phi \colon L\to \sff(S, M)$, there exists a Borel map $\varphi \colon L\times S\to M$ such that for any $x\in L$, we have $\varphi(x, s)=(\Phi(x))(s)$ for a.e.\ $s\in S$.
\end{enumerate}
\end{lem}

We refer to \cite[Lemmas 2.6 and 2.7]{fmw} for these assertions (i) and (ii), respectively.

\begin{proof}[Proof of Proposition \ref{prop-comp}]
As already mentioned, by \cite[Theorem 1]{fm1}, we have a Borel action of a discrete group $G$ on $X$ such that
\[\calr =\{\, (gx, x)\in X\times X\mid g\in G,\ x\in X\, \}.\]
Applying the Rohlin decomposition theorem (\cite[Proposition 2.21]{fmw}), we may assume that there exist a standard probability space $(S, \eta)$, a Borel action of $G$ on $Z\times S$ and a Borel isomorphism $\Theta \colon X\to Z\times S$ such that
\begin{itemize}
\item $\Theta_*\mu$ and $\xi \times \eta$ are equivalent; and
\item the actions of $G$ on $X$ and on $Z\times S$ are equivariant under $\Theta$; and
\item for any $z\in Z$, the measure on $\{ z\} \times S$ induced by $\eta$ is quasi-invariant and ergodic under the action of $G$.
\end{itemize}
We identify $X$ with $Z\times S$ under $\Theta$.
There exists a Borel map $\rho_S\colon G\times Z\times S\to S$ such that for any $g\in G$ and any $z\in Z$, we have $g(z, s)=(z, \rho_S(g, z, s))$ for a.e.\ $s\in S$.
For any $z\in Z$, there exists a Borel map $\rho_z\colon G\times S\to S$ such that $\rho_z(g, s)=\rho_S(g, z, s)$ for any $g\in G$ and any $s\in S$.

For any $z\in Z$, we have the action $\rho_z$ of $G$ on a conull Borel subset of $S$.
The measure $\eta$ is quasi-invariant under this action.
We denote by $\calg_z$ the discrete measured groupoid associated with the action $\rho_z$.

Let $m$ be a non-negative integer.
We denote by $G^m$ the direct product of $m$ copies of $G$ if $m\geq 1$, and the trivial group if $m=0$.
For any $z\in Z$, the space $C^m(\calg_z, C)$ is identified with $\sff(G^m\times S, C)$ under the measurable isomorphism from $G^m\times S$ onto $\calg_z^{(m)}$ defined by
\[((g_1,\ldots, g_m), s)\mapsto ((g_1, \rho_z(g_2\cdots g_m, s)), (g_2, \rho_z(g_3\cdots g_m, s)),\ldots, (g_m, s))\]
for $g_1,\ldots, g_m\in G$ and $s\in S$.
For any $z\in Z$, we have the map
\[I_{m, z}\colon C^m(\calr_z, C)\to \sff(G^m\times S, C)\]
induced by the quotient map from $\calg_z$ onto $\calr_z$.
We set $\Omega_{m, z}=I_{m, z}(C^m(\calr_z, C))$.

\begin{lem}
Let $m$ be a non-negative integer.
Then the set
\[\Omega_m=\{ \, (z, f)\in Z\times \sff(G^m\times S, C)\mid f\in \Omega_{m, z}\, \}\]
is Borel in $Z\times \sff(G^m\times S, C)$.
\end{lem}

\begin{proof}
For any $z\in Z$, we have $\Omega_{0, z}=\sff(S, C)$, and therefore have $\Omega_0=Z\times \sff(S, C)$.
The lemma follows when $m=0$.

We assume $m\geq 1$.
We claim that for any $\gamma \in G$ and any $i=1,\ldots, m$, there exists a Borel map
\[\tau_\gamma^i \colon Z\times \sff(G^m\times S, C)\to \sff(G^m\times S, C)\]
such that for any $z\in Z$ and any $f\in \sff(G^m\times S, C)$, the following equation holds:
For any $g=(g_1,\ldots, g_m)\in G^m$ and a.e.\ $s\in S$, we have
\begin{align*}
& \tau_\gamma^i(z, f)(g, s)\\
= \, & \begin{cases}
f((g_1,\ldots, g_{i-1}, \gamma g_i, g_{i+1},\ldots, g_m), s) & \textrm{if}\ \rho_z(\gamma, \rho_z(g_i\cdots g_m, s))=\rho_z(g_i\cdots g_m, s)\\
f(g, s) & \textrm{otherwise}.
\end{cases}
\end{align*}
If the claim is shown, then the lemma follows because we have
\[\Omega_m=\bigcap_{\gamma \in G}\bigcap_{i=1}^m\, \{ \, (z, f)\in Z\times \sff(G^m\times S, C)\mid \tau_\gamma^i(z, f)=f\, \}.\]

Fix $\gamma \in G$ and $i=1,\ldots, m$.
Applying Lemma \ref{lem-map} (ii) to the identity on $\sff(G^m\times S, C)$, we obtain a Borel map
\[u\colon \sff(G^m\times S, C)\times G^m\times S\to C\]
such that for any $f\in \sff(G^m\times S, C)$, we have $u(f, t)=f(t)$ for a.e.\ $t\in G^m\times S$.
Similarly, we obtain a Borel map
\[v\colon \sff(G^m\times S, C)\times G^m\times S\to C\]
such that for any $f\in \sff(G^m\times S, C)$, we have
\[v(f, (g, s))=f((g_1,\ldots, g_{i-1}, \gamma g_i, g_{i+1},\ldots, g_m), s)\]
for any $g=(g_1,\ldots, g_m)\in G^m$ and a.e.\ $s\in S$.
We define a map
\[w\colon Z\times \sff(G^m\times S, C)\times G^m\times S\to C\]
by
\[w(z, f, (g, s))=\begin{cases}
v(f, (g, s)) & \textrm{if}\ \rho_z(\gamma, \rho_z(g_i\cdots g_m, s))=\rho_z(g_i\cdots g_m, s)\\
u(f, (g, s)) & \textrm{otherwise}.
\end{cases}
\]
for $z\in Z$, $f\in \sff(G^m\times S, C)$, $g=(g_1,\ldots, g_m)\in G^m$ and $s\in S$.
The map $w$ is Borel, and induces the Borel map $\widetilde{w}$ by Lemma \ref{lem-map} (i).
We set $\tau_\gamma^i=\widetilde{w}$.
The claim then follows.
\end{proof}

We set $\calg =G\ltimes (X, \mu)$.
For a non-negative integer $m$, the space $C^m(\calg, C)$ is identified with $\sff(G^m\times X, C)$ under the Borel isomorphism from $G^m\times X$ onto $\calg^{(m)}$ defined by
\[((g_1,\ldots, g_m), x)\mapsto ((g_1, g_2\cdots g_mx), (g_2, g_3\cdots g_mx),\ldots, (g_m, x))\]
for $g_1,\ldots, g_m\in G$ and $x\in X$.
We have the map
\[I_m\colon C^m(\calr, C)\to \sff(G^m\times X, C)\]
induced by the quotient map from $\calg$ onto $\calr$.
For each $\phi \in C^m(\calr, C)$, we identify it with $I_m(\phi)$ if there is no confusion.
We then have the induced map $\widetilde{\phi}\colon Z\to \sff(G^m\times S, C)$.

The rest of the proof relies on the proof of \cite[Theorem 3.4 and Corollary 3.6]{fmw}.
Fix a non-negative integer $n$, and pick a measurable map $\phi \colon \calr^{(n)}\to C$.
Suppose that for a.e.\ $z\in Z$, the map $\phi_z$ belongs to $B^n(\calr_z, C)$.
To prove the proposition, we may assume $n\geq 1$ because $B^0=0$.
We define a map
\[d\colon Z\times \sff(G^n\times S, C)\times \sff(G^{n-1}\times S, C)\to \sff(G^n\times S, C)\]
by
\[d(z, \psi, f)(g, s)=\psi(g, s)-(\delta_z^{n-1}f)(g, s)\]
for $z\in Z$, $\psi \in \sff(G^n\times S, C)$, $f\in \sff(G^{n-1}\times S, C)$, $g\in G^n$ and $s\in S$, where
\[\delta_z^{n-1}\colon C^{n-1}(\calg_z, C)\to C^n(\calg_z, C)\]
is the coboundary map.
Following the proof of \cite[Theorem 3.4]{fmw}, we see that the map $d$ is Borel.
We define a map
\[p\colon Z\times \sff(G^{n-1}\times S, C)\to Z\times \sff(G^n\times S, C)\]
by
\[p(z, f)=(z, d(z, \widetilde{\phi}(z), f))\]
for $z\in Z$ and $f\in \sff(G^{n-1}\times S, C)$.
The map $p$ is Borel.

We define $\Sigma$ as the Borel subset $Z\times \{ 0\}$ of $Z\times \sff(G^n\times S, C)$.
The assumption that $\phi_z$ belongs to $B^n(\calr_z, C)$ for a.e.\ $z\in Z$ implies that the projection of $p^{-1}(\Sigma)\cap \Omega_{n-1}$ into $Z$ contains a conull Borel subset of $Z$.
By the von Neumann selection theorem (\cite[Theorem 3.4.3]{arv}, \cite[Theorem 2.3]{fmw}), there exists a Borel map $F\colon Z\to \sff(G^{n-1}\times S, C)$ such that
\[(z, F(z))\in p^{-1}(\Sigma)\cap \Omega_{n-1}\quad \textrm{for\ a.e.}\ z\in Z.\]
By Lemma \ref{lem-map} (ii), there exists a Borel map $f\colon G^{n-1}\times X\to C$ such that for any $z\in Z$ and any $g\in G^{n-1}$, we have
\[F(z)(g, s)=f(g, (z, s))\quad \textrm{for\ a.e.}\ s\in S.\]
The map $f$ belongs to $I_{n-1}(C^{n-1}(\calr, C))$ because $(z, F(z))$ belongs to $\Omega_{n-1}$ for a.e.\ $z\in Z$.
For a.e.\ $z\in Z$, we have $d(z, \widetilde{\phi}(z), F(z))=0$ because $(z, F(z))$ belongs to $p^{-1}(\Sigma)$.
For a.e.\ $x=(z, s)\in Z\times S$ and any $g\in G^n$, we therefore have
\[\phi(g, x)=\widetilde{\phi}(z)(g, s)=(\delta_z^{n-1}F(z))(g, s)=(\delta^{n-1}f)(g, x).\]
It follows that $\phi \in B^n(\calr, C)$.
\end{proof}

It follows from the definition of treeability that for any treeable equivalence relation, its almost every ergodic component is also treeable.
Combining this with Corollary \ref{cor-tree} and Proposition \ref{prop-comp}, we obtain the following:

\begin{cor}\label{cor-treeable}
Let $\calr$ be a discrete measured equivalence relation which is p.m.p.\ and treeable (but not necessarily ergodic).
Let $C$ be an abelian Polish group, and regard it as an $\calr$-module on which $\calr$ acts trivially.
Then $H^n(\calr, C)=0$ for any integer $n$ with $n\geq 2$.
\end{cor}


\section{Splitting of groupoids}\label{sec-split}

We introduce a central subgroupoid of a discrete measured groupoid in the following:

\begin{defn}\label{defn-ce}
Let $\calg$ be a discrete measured groupoid on a standard probability space $(X, \mu)$.
We mean by a {\it central subgroupoid} of $\calg$ the pair $(C, \iota)$ of a discrete group $C$ and an injective homomorphism $\iota \colon \cal{C}\to \calg$, where $\cal{C}$ is the discrete measured groupoid associated with the trivial action of $C$ on $X$, such that
\begin{itemize}
\item the map from the unit space of $\cal{C}$ into that of $\calg$ induced by $\iota$ is the identity on $X$; and
\item for any $c\in C$ and a.e.\ $g\in \calg$, the equation $g\iota(c, s(g))=\iota(c, r(g))g$ holds.
\end{itemize}
\end{defn}

For $c\in C$ and $x\in X$, we write $\iota(c, x)$ as $c$ simply if its range or source is understood from the context.
We then have the equation $gc=cg$ for any $c\in C$ and a.e.\ $g\in \calg$.
This notation satisfies the following associative law:
For any $c\in C$ and a.e.\ $(g, h)\in \calg^{(2)}$, we have $c(gh)=(cg)h=g(hc)=(gh)c$.
This element can thus be denoted by $cgh=gch=ghc$ without ambiguity.
Let $\calg /C$ denote the quotient space for right (or left) multiplication of $C$ on $\calg$.
The space $\calg /C$ naturally admits a structure of a discrete measured groupoid on $(X, \mu)$.
We then say that $\calg$ is a {\it central extension} of $\calg /C$ with a central subgroup $C$.

Let $G$ be a discrete group with a central subgroup $C$.
Let $G\c (X, \mu)$ be a non-singular action such that $C$ acts trivially.
We define $\iota$ as the identity on $C\ltimes (X, \mu)$.
The pair $(C, \iota)$ is then a central subgroupoid of $G\ltimes (X, \mu)$. 
Theorem \ref{thm-split-groupoid} is therefore a special case of the following:

\begin{thm}\label{thm-section}
Let $\calg$ be a p.m.p.\ discrete measured groupoid on a standard probability space $(X, \mu)$ and $(C, \iota)$ a central subgroupoid of $\calg$.
Suppose that $\calg /C$ is isomorphic to a treeable equivalence relation.
Let $\theta \colon \calg \to \calg /C$ denote the quotient map.
Then the following two assertions hold:
\begin{enumerate}
\item The homomorphism $\theta$ splits, that is, there exists a subgroupoid $\calg_1$ of $\calg$ such that the restriction of $\theta$ to $\calg_1$ is an isomorphism onto $\calg/C$.
\item The product set $C\times \calg_1$ admits the structure of a discrete measured groupoid on $(X, \mu)$ with respect to the coordinatewise product operation.
The map $I\colon C\times \calg_1\to \calg$ defined by $I(c, g)=cg$ for $c\in C$ and $g\in \calg_1$ is then an isomorphism of discrete measured groupoids.
\end{enumerate}
\end{thm}

\begin{proof}
We set $\calq =\calg /C$, and pick a measurable section $u\colon \calq \to \calg$ of $\theta$.
We define a map $\sigma \colon \calq^{(2)} \to C$ so that
\[\sigma(g, h)u(gh)=u(g)u(h)\quad \textrm{for\ a.e.\ }(g, h)\in \calq^{(2)}.\]
The map $\sigma$ belongs to $Z^2(\calq, C)$.
By Corollary \ref{cor-treeable}, there exists a measurable map $\phi \colon \calq \to C$ such that
\[\sigma(g, h)=\phi(h)\phi(gh)^{-1}\phi(g)\quad \textrm{for\ a.e.\ }(g, h)\in \calq^{(2)}.\]

We set
\[\calg_1=\{ \, \phi(g)^{-1}u(g)\in \calg \mid g\in \calq \, \}.\]
We show that $\calg_1$ is a subgroupoid of $\calg$.
For a.e.\ $(g, h)\in \calq^{(2)}$, the equation
\begin{align}\label{eq-gh}
(\phi(g)^{-1}u(g))(\phi(h)^{-1}u(h))&=\phi(g)^{-1}\phi(h)^{-1}u(g)u(h)=\phi(g)^{-1}\phi(h)^{-1}\sigma(g, h)u(gh)\\
&=\phi(gh)^{-1}u(gh)\nonumber
\end{align}
holds.
The subset $\calg_1$ is thus closed under multiplication.
For $x\in X$, let $e_x\in \calq$ denote the unit at $x$.
By equation (\ref{eq-gh}), for a.e.\ $x\in X$, the element $\phi(e_x)^{-1}u(e_x)$ is the unit of $\calg$ at $x$.
It also follows from equation (\ref{eq-gh}) that for a.e.\ $g\in \calq$, the element $\phi(g^{-1})^{-1}u(g^{-1})$ is the inverse of $\phi(g)^{-1}u(g)$.
The subset $\calg_1$ is therefore a subgroupoid of $\calg$.

By the definition of $\calg_1$, the restriction of $\theta$ to $\calg_1$ is an isomorphism onto $\calq$.
Assertion (i) follows.
We have $cg=gc$ for any $c\in C$ and $g\in \calg$.
The map $I\colon C\times \calg_1\to \calg$ in assertion (ii) therefore preserves products, and is an isomorphism.
Assertion (ii) follows.
\end{proof}

\begin{rem}
Series \cite{series} constructs cohomology theory of a discrete measured equivalence relation $\calr$ with coefficient in a measured field $A=\{ A_x\}_{x\in X}$ of abelian, locally compact and second countable groups.
She studies a groupoid extension of $\calr$ by $A$ and its connection with cohomology, following the theory of group extensions.
A central extension in Definition \ref{defn-ce} is an extension in \cite{series} such that the field of groups is constant and further admits centrality.
Among other things, Series aims to reconstruct a groupoid from given $\calr$ and $A$.
The reconstruction is described through cohomology of $\calr$ (\cite[Theorem 3.5]{series}).
Using this, she also discusses when an extension splits (\cite[\S 4, (B)]{series}).
\end{rem}

Let $G\c (X, \mu)$ and $H\c (Y, \nu)$ be ergodic and p.m.p.\ actions of discrete groups.
These two actions are called {\it orbit equivalent (OE)} if the equivalence relations associated to them are isomorphic.
The following is a consequence of Theorem \ref{thm-section}, and asserts splitting of a free action of a central extension of a treeable group.

\begin{thm}\label{thm-split}
Let $1\to C\to G\to \Gamma \to 1$ be an exact sequence of discrete groups such that $C$ is central in $G$.
Let $C\c (Y, \nu)$ and $\Gamma \c (Z, \xi)$ be ergodic, free and p.m.p.\ actions.
Suppose that the equivalence relation associated with the action $\Gamma \c (Z, \xi)$ is treeable.
Let $C\times \Gamma$ act on the product space $(Y\times Z, \nu \times \xi)$ coordinatewise.

Then there exist an ergodic, free and p.m.p.\ action $G\c (X, \mu)$ and an isomorphism of measure spaces, $f\colon (X, \mu)\to (Y\times Z, \nu \times \xi)$, such that
\begin{itemize}
\item the two actions $G\c (X, \mu)$ and $C\times \Gamma \c (Y\times Z, \nu \times \xi)$ are OE through $f$;
\item for any $c\in C$ and a.e.\ $x\in X$, we have $f(cx)=(c, e)f(x)$, where $e$ is the neutral element of $\Gamma$;
\item for any $g\in G$ and a.e.\ $x\in X$, we have $f(gx)=(c, q(g))f(x)$ for some $c\in C$, where $q\colon G\to \Gamma$ is the quotient map. 
\end{itemize}
\end{thm}

\begin{proof}
Let $G$ act on $(Z, \xi)$ through the map $q$.
We set $\calg =G\ltimes (Z, \xi)$ and $\calq =\Gamma \ltimes (Z, \xi)$.
By Theorem \ref{thm-section}, we have a subgroupoid $\calg_1$ of $\calg$ and an isomorphism $I\colon C\times \calg_1\to \calg$ such that $I(c, g)=cg$ for any $c\in C$ and a.e.\ $g\in \calg_1$.
Let $\theta_1\colon \calg_1\to \calq$ be the restriction of the quotient map from $\calg$ onto $\calq$.
The map $\theta_1$ is an isomorphism.

We also define the p.m.p.\ action $C\times \Gamma \c (Z, \xi)$ through the projection from $C\times \Gamma$ onto $\Gamma$, and set $\calh =(C\times \Gamma)\ltimes (Z, \xi)$.
We then have the isomorphism $J\colon \calh \to \calg$ defined by
\[J((c, \gamma), z)=I(c, \theta_1^{-1}(\gamma, z))\quad \textrm{for}\ c\in C,\ \gamma \in \Gamma \ \textrm{and}\ z\in Z.\]

We construct a measure space $(\Sigma, m)$ as follows:
We set $\Sigma =Y\times Z\times G$, and define $m$ as the product measure of $\nu$, $\xi$ and the counting measure on $G$.
Let $\alpha \colon \calh \to G$ be the homomorphism defined as the composition of $J$ with the projection from $\calg$ onto $G$.
We define a measure-preserving action $(C\times \Gamma)\times G\c (\Sigma, m)$ by
\[((c, \gamma), g)(y, z, h)=(cy, \gamma z, \alpha((c, \gamma), z)hg^{-1})\]
for $c\in C$, $\gamma \in \Gamma$, $g, h\in G$, $y\in Y$ and $z\in Z$.
We denote by $1_{C\times \Gamma}$ and $1_G$ the neutral elements of $C\times \Gamma$ and $G$, respectively.
The subset $Y\times Z\times \{ 1_G\}$ of $\Sigma$ is a fundamental domain for both the actions of $(C\times \Gamma) \times \{ 1_G\}$ and $\{ 1_{C\times \Gamma}\} \times G$ on $(\Sigma, m)$.
It follows that $(\Sigma, m)$ is a measure-equivalence coupling of $C\times \Gamma$ and $G$.

We set $X=Y\times Z\times \{ 1_G\}$.
We have the natural p.m.p.\ action of $G$ on $X$ because $X$ is identified with the quotient space $\Sigma /((C\times \Gamma)\times \{ 1_G\})$.
We also have the p.m.p.\ action of $C\times \Gamma$ on $X$ because $X$ is identified with the quotient space $\Sigma /(\{ 1_{C\times \Gamma}\} \times G)$.
This action is isomorphic to the coordinatewise action $C\times \Gamma \c Y\times Z$ under the projection from $X$ onto $Y\times Z$.
The two actions $G\c X$ and $C\times \Gamma \c X$ are OE under the identity map on $X$.
Since the latter action is ergodic and free, so is the former action.
For any $c\in C$, the automorphism of $X$ induced by $c$ as an element of $G$ and that induced by $(c, e)$ are equal.
The second condition in Theorem \ref{thm-split} holds.
The third condition holds because for any $c\in C$, any $\gamma \in \Gamma$ and a.e.\ $z\in Z$, we have $q\circ \alpha((c, \gamma), z)=\gamma$.
\end{proof}


\section{The Haagerup property of discrete measured groupoids}\label{sec-hap}

Throughout this section, we fix a standard probability space $(X, \mu)$.

\subsection{Hilbert bundles and representations}

We review a Hilbert bundle over $X$ and a representation of a groupoid on it, following \cite[Sections 2 and 3]{ana-t}, \cite[Section 2]{ramsay} and \cite[Appendix F]{wil}.

Let $\sh =\{ \sh_x\}_{x\in X}$ be a family of separable Hilbert spaces indexed by elements of $X$.
We define $X\ast \sh$ as the set of pairs $(x, v)$ of $x\in X$ and $v\in \sh_x$, and define $p\colon X\ast \sh \to X$ as the projection sending each $(x, v)\in X\ast \sh$ to $x$.
We often identify the fiber $p^{-1}(x)=\{ x\} \times \sh_x$ with $\sh_x$ if there is no confusion.
A map $\sigma \colon X \to X\ast \sh$ with $p(\sigma_x)=x$ for any $x\in X$ is called a {\it section} of $X\ast \sh$.
A section $\sigma$ is called {\it normalized} if $\Vert \sigma_x\Vert =1$ for any $x\in X$.

\begin{defn}[\ci{Definition 2.2}{ana-t}, \ci{p.264}{ramsay}, \ci{Definition F.1}{wil}]\label{defn-hb}
Let $\sh =\{ \sh_x\}_{x\in X}$ be a family of separable Hilbert spaces with $p\colon X\ast \sh \to X$ the projection defined above.
We mean by a {\it Hilbert bundle} over $X$ is such a set $X\ast \sh$ equipped with a standard Borel structure satisfying the following conditions (1) and (2):
\begin{enumerate}
\item[(1)] A subset $E$ of $X$ is Borel if and only if $p^{-1}(E)$ is Borel in $X\ast \sh$.
\item[(2)] There exists a sequence $(\sigma^n)_n$ of Borel sections of $X\ast \sh$ such that
\begin{enumerate}
\item[(a)] for any $n$, the function $\bar{\sigma}^n$ on $X\ast \sh$ defined by $\bar{\sigma}^n(x, v)=\langle \sigma^n_x, v\rangle$ for $(x, v)\in X\ast \sh$ is Borel, where $\langle \cdot, \cdot \rangle$ is inner product of $\sh_x$;
\item[(b)] for any $n$ and $m$, the function $x\mapsto \langle \sigma^n_x, \sigma^m_x\rangle$ on $X$ is Borel; and
\item[(c)] the family of the functions $\bar{\sigma}^n$ for $n$ and the map $p$ separates points of $X\ast \sh$.
\end{enumerate}
\end{enumerate}
The sequence $(\sigma^n)_n$ is called a {\it fundamental sequence} for $X\ast \sh$.
\end{defn}

It is known that a section $\sigma \colon X\to X\ast \sh$ is Borel if and only if for any $n$, the function $x\mapsto \langle \sigma_x, \sigma^n_x\rangle$ on $X$ is Borel (\cite[Remark F.3]{wil}).
For a Hilbert bundle, we use the symbol $X\ast \sh$ or simply $\sh$ if there is no confusion.
Unless otherwise mentioned, a section of a Hilbert bundle is assumed to be Borel.
For a given family $\sh =\{ \sh_x\}_{x\in X}$, to equip the set $X\ast \sh$ with a standard Borel structure, we use the following:

\begin{prop}[\ci{Proposition F.8}{wil}]\label{prop-hb}
Suppose that we have a family $\sh =\{ \sh_x\}_{x\in X}$ of separable Hilbert spaces indexed by elements of $X$ and have a sequence $(\sigma^n)_n$ of sections of $X\ast \sh$ such that conditions (b) and (c) in Definition \ref{defn-hb} hold.
Then there exists a unique standard Borel structure on $X\ast \sh$ such that $X\ast \sh$ is a Hilbert bundle over $X$ and $(\sigma^n)_n$ is its fundamental sequence. 
\end{prop}

Indeed, in view of the proof of loc.\ cit., the standard Borel structure on $X\ast \sh$ is defined as the weakest measurable structure such that the functions $\bar{\sigma}^n$ for $n$ and the projection $p\colon X\ast \sh \to X$ are measurable, where $\bar{\sigma}^n$ is defined similarly to that in condition (a) in Definition \ref{defn-hb}.

\begin{rem}
By Proposition \ref{prop-hb}, given a countable family $(\sh^n)_n$ of Hilbert bundles over $X$, we naturally obtain their direct sum $\bigoplus_n\sh^n$, the Hilbert bundle over $X$ whose fiber at $x\in X$ is the direct sum $\bigoplus_n\sh^n_x$.
\end{rem}

For a Hilbert bundle $\sh$ over $X$, we denote by ${\rm Iso}(\sh)$ the groupoid on $X$ whose elements are triples $(x, V, y)$ with $x, y\in X$ and $V$ a Hilbert space isomorphism from $\sh_y$ onto $\sh_x$, and whose product is defined by $(x, V, y)(y, W, z)=(x, V\circ W, z)$.
We equip ${\rm Iso}(\sh)$ with the weakest measurable structure such that for any $n$ and $m$, the function $(x, V, y)\mapsto \langle V\sigma^n_y, \sigma^m_x\rangle$ on ${\rm Iso}(\sh)$ is measurable, where $(\sigma^n)_n$ is a fundamental sequence for $\sh$.

Let $\calg$ be a discrete measured groupoid on $(X, \mu)$.
We mean by a {\it representation} of $\calg$ on a Hilbert bundle $\sh$ over $X$ a measurable homomorphism $\pi \colon \calg \to {\rm Iso}(\sh)$ inducing the identity on $X$.
For an element $g\in \calg$ whose range and source are $y$ and $x$, respectively, the element $\pi (g)\in {\rm Iso}(\sh)$ is written as $(y, \bar{\pi}(g), x)$, where $\bar{\pi}(g)$ is a Hilbert space isomorphism from $\sh_x$ onto $\sh_y$.
To lighten the symbols, we identify $\pi(g)$ with $\bar{\pi}(g)$.
We then have $\pi(gh)=\pi(g)\pi(h)$ for a.e.\ $(g, h)\in \calg^{(2)}$, and have $\pi(g^{-1})=\pi(g)^{-1}$ for a.e.\ $g\in \calg$.

In the rest of this subsection, we discuss the following two properties of representations of $\calg$, the $c_0$-property and the weak containment of the trivial representation.

\begin{defn}
Let $\calg$ be a discrete measured groupoid on $(X, \mu)$ and $\pi$ a representation of $\calg$ on a Hilbert bundle $\sh$ over $X$.
We say that $\pi$ is $c_0$ (or is a {\it $c_0$-representation}) if for any normalized sections $u$, $v$ of $\sh$, any $\delta>0$ and any $\varepsilon >0$, there exists a measurable subset $Y$ of $X$ such that $\mu(X\setminus Y)<\delta$ and for a.e.\ $x\in Y$, we have
\[|\{ \, g\in (\calg|_Y)_x\mid |\langle \pi(g)u_x, v_{r(g)}\rangle |>\varepsilon \, \}|<\infty.\]
\end{defn}

\begin{lem}\label{lem-c0-pre-check}
Let $\calg$ be a discrete measured groupoid on $(X, \mu)$ and $\pi$ a representation of $\calg$ on a Hilbert bundle $\sh$ over $X$.
Then $\pi$ is $c_0$ if there exists a countable family $(e^n)_{n\in \N}$ of sections of $\sh$ satisfying the following conditions (1) and (2):
\begin{enumerate}
\item[(1)] For a.e.\ $x\in X$, the subset $(e^n_x)_{n\in \N}$ is total in $\sh_x$.
\item[(2)] For any $n, m\in \N$, any $\delta>0$ and any $\varepsilon >0$, there exists a measurable subset $Y$ of $X$ such that $\mu(X\setminus Y)<\delta$ and for a.e.\ $x\in Y$, we have
\[|\{ \, g\in (\calg|_Y)_x\mid |\langle \pi(g)e^n_x, e^m_{r(g)}\rangle |>\varepsilon \, \}|<\infty.\]
\end{enumerate}
\end{lem}

\begin{proof}
We first prove that $\pi$ is $c_0$ if $(e^n)_{n\in \N}$ further satisfies the following conditions (a) and (b):
\begin{enumerate}
\item[{\rm (a)}] For a.e.\ $x\in X$ and any $n\in \N$, we have either $\Vert e^n_x\Vert =1$ or $e^n_x=0$.
\item[{\rm (b)}] For a.e.\ $x\in X$, the unit vectors in the set $(e^n_x)_{n\in \N}$ form an orthonormal basis of $\sh_x$.
\end{enumerate}
Pick normalized sections $u$, $v$ of $\sh$ and numbers $\delta>0$ and $\varepsilon >0$.
By conditions (a) and (b), there exist a measurable subset $Y_1$ of $X$ and $N\in \N$ such that $\mu(X\setminus Y_1)<\delta /2$ and for a.e.\ $x\in Y_1$, we have
\[\left\Vert u_x-\sum_{n=1}^N\langle u_x, e^n_x\rangle e^n_x\right\Vert <\frac{\, \varepsilon \, }{4}\quad \textrm{and}\quad \left\Vert v_x-\sum_{n=1}^N\langle v_x, e^n_x\rangle e^n_x\right\Vert <\frac{\, \varepsilon \, }{4}.\]
By condition (2), there exists a measurable subset $Y_2$ of $X$ such that $\mu(X\setminus Y_2)<\delta /2$ and for any $n, m=1, 2,\ldots, N$ and a.e.\ $x\in Y_2$, we have
\[|\{ \, g\in (\calg|_{Y_2})_x\mid |\langle \pi(g)e^n_x, e^m_{r(g)}\rangle |>\varepsilon /(2N^2)\, \}|<\infty.\]
We set $Y=Y_1\cap Y_2$.
We have $\mu(X\setminus Y)<\delta$.
For a.e.\ $g\in \calg|_Y$, we have
\[|\langle \pi(g)u_{s(g)}, v_{r(g)}\rangle|<\frac{\, \varepsilon \, }{2}+\sum_{n, m=1}^N|\langle \pi(g)e_{s(g)}^n, e_{r(g)}^m\rangle |.\]
For a.e.\ $x\in Y$, we also have
\begin{align*}
& \{ \, g\in (\calg|_Y)_x \mid |\langle \pi(g)u_x, v_{r(g)}\rangle |>\varepsilon \, \} \\
\subset \ & \bigcup_{n, m=1}^N\{ \, g\in (\calg|_Y)_x \mid |\langle \pi(g)e_x^n, e_{r(g)}^m\rangle |>\varepsilon /(2N^2)\, \}.
\end{align*}
It follows that the set in the left hand side is finite.

We proved that $\pi$ is $c_0$ if $(e^n)_{n\in \N}$ satisfies conditions (a) and (b).
Suppose that $(e^n)_{n\in \N}$ satisfies only conditions (1) and (2).
Let $(f^n)_{n\in \N}$ be the sequence of sections of $\sh$ obtained by applying the Gram-Schmidt process to the sequence $(e^n)_{n\in \N}$.
The sequence $(f^n)_{n\in \N}$ then satisfies conditions (a) and (b).

We check that $(f^n)_{n\in \N}$ satisfies condition (2).
Fix $n, m\in \N$.
Pick $\delta >0$ and $\varepsilon >0$.
There exist measurable functions $\alpha^k, \beta^l\colon X\to \C$ indexed by $k=1,\ldots, n$ and $l=1,\ldots, m$ such that for a.e.\ $x\in X$, we have $f^n_x=\sum_{k=1}^n\alpha^k_xe^k_x$ and $f^m_x=\sum_{l=1}^m\beta^l_xe^l_x$.
By condition (2) for $(e^n)_{n\in \N}$, there exist a measurable subset $Y$ of $X$ and a number $M>0$ such that $\mu(X\setminus Y)<\delta$ and for any $k=1,\ldots, n$, any $l=1,\ldots, m$ and a.e.\ $x\in Y$, we have $|\alpha^k_x|\leq M$, $|\beta^l_x|\leq M$ and
\[|\{ \, g\in (\calg|_Y)_x\mid |\langle \pi(g)e^k_x, e^l_{r(g)}\rangle|> \varepsilon /(nmM^2)\, \}|<\infty.\]
For a.e.\ $x\in Y$, we then have
\begin{align*}
& \{ \, g\in (\calg|_Y)_x\mid |\langle \pi(g)f^n_x, f^m_{r(g)}\rangle|>\varepsilon \, \} \\
\subset \, & \bigcup_{k=1}^n\bigcup_{l=1}^m\, \{ \, g\in (\calg|_Y)_x\mid |\alpha^k_x\overline{\beta^l_{r(g)}}\langle \pi(g)e^k_x, e^l_{r(g)}\rangle|>\varepsilon /(nm)\, \} \\
\subset \, & \bigcup_{k=1}^n\bigcup_{l=1}^m\, \{ \, g\in (\calg|_Y)_x\mid |\langle \pi(g)e^k_x, e^l_{r(g)}\rangle|>\varepsilon /(nmM^2)\, \}.
\end{align*}
It follows that the set in the left hand side is finite.
Condition (2) for $(f^n)_{n\in \N}$ is proved.

The sequence $(f^n)_{n\in \N}$ satisfies conditions (a), (b) and (2).
The lemma is thus reduced to the case discussed in the first paragraph of the proof.
\end{proof}

\begin{lem}\label{lem-c0-check}
Let $\calg$ be a discrete measured groupoid on $(X, \mu)$ and $\pi$ a representation of $\calg$ on a Hilbert bundle $\sh$ over $X$.
Then $\pi$ is $c_0$ if there exists a countable family $(e^n)_{n\in \N}$ of sections of $\sh$ satisfying the following conditions (I) and (II):
\begin{enumerate}
\item[(I)] For a.e.\ $x\in X$, the subset $\{ \, \pi(g)e^n_{s(g)}\in \sh_x \mid g\in \calg^x,\ n\in \N \, \}$ is total in $\sh_x$.
\item[(II)] For any $n, m\in \N$, any $\delta>0$ and any $\varepsilon >0$, there exists a measurable subset $Y$ of $X$ such that $\mu(X\setminus Y)<\delta$ and for a.e.\ $x\in Y$, we have
\[|\{ \, g\in (\calg|_Y)_x\mid |\langle \pi(g)e^n_x, e^m_{r(g)}\rangle |>\varepsilon \, \}|<\infty.\]
\end{enumerate}
\end{lem}

\begin{proof}
Pick a countable family $(\phi_k)_{k\in \N}$ of sections $\phi_k\colon X\to \calg$ of the range map $r$ such that the equation $\calg=\bigcup_{k\in \N}\phi_k(X)$ holds and for any $k\in \N$, the map $s\circ \phi_k\colon X\to X$ is an automorphism.
For any $k, l\in \N$, we have the measurable isomorphism $U_k^l\colon \calg\to \calg$ defined by
\[U_k^l(g)=\phi_l(r(g))^{-1}g\phi_k(s(g))\quad  \textrm{for}\ g\in \calg.\]
For $n, k\in \N$, we define a section $e^{n, k}\colon X\to \sh$ of $\sh$ by
\[e^{n, k}_x=\pi(\phi_k(x))e^n_{s\circ \phi_k(x)}\quad \textrm{for}\ x\in X.\]
We show that the family $(e^{n, k})_{n, k}$ satisfies conditions (1) and (2) in Lemma \ref{lem-c0-pre-check}.
Condition (1) follows from condition (I).

To prove that condition (2) holds, we fix $n, m, k, l\in \N$ and pick $\delta >0$ and $\varepsilon >0$.
By condition (II), there exists a measurable subset $Z$ of $X$ such that $\mu(X\setminus Y)<\delta$, where we set
\[Y= \{ \, x\in X \mid s\circ \phi_k(x)\in Z,\ s\circ \phi_l(x)\in Z \, \};\]
and for a.e.\ $x\in Z$, we have
\begin{equation}\label{phi-finite}
|\{ \, g\in (\calg|_Z)_x\mid |\langle \pi(g)e^n_x, e^m_{r(g)}\rangle |>\varepsilon  \, \}|<\infty.
\end{equation}
For a.e.\ $g\in \calg$, putting $y=r(g)$ and $x=s(g)$, we have
\[\langle \pi(g)e^{n, k}_x, e^{m, l}_y\rangle =\langle \pi(U_k^l(g))e^n_{s\circ \phi_k(x)}, e^m_{s\circ \phi_l(y)}\rangle.\]
It follows that for a.e.\ $x\in Y$, we have
\begin{align*}
& \{ \, g\in (\calg|_Y)_x\mid  |\langle \pi(g)e^{n, k}_x, e^{m, l}_{r(g)}\rangle | >\varepsilon \, \} \\
= \, & \{ \, g\in (\calg|_Y)_x \mid  |\langle \pi(U_k^l(g))e^n_{s\circ \phi_k(x)}, e^m_{s\circ \phi_l(r(g))}\rangle|>\varepsilon \, \}
\end{align*}
For any element $g$ in the right hand side, we have $U_k^l(g)\in \calg|_Z$.
By (\ref{phi-finite}), the right hand side is therefore finite.
Condition (2) for $(e^{n, k})_{n, k}$ follows.
By Lemma \ref{lem-c0-pre-check}, $\pi$ is $c_0$.
\end{proof}

We next recall the weak containment of the trivial representation.
This is used to define amenability and property (T) of a discrete measured groupoid (\cite{ana-t}, \cite{ar-book}).

\begin{defn}
Let $\calg$ be a discrete measured groupoid on $(X, \mu)$ and $\pi$ a representation of $\calg$ on a Hilbert bundle $\sh$ over $X$.
A sequence $(v^n)_n$ of normalized sections of $\sh$ is called {\it almost invariant} under $\pi(\calg)$ if for a.e.\ $g\in \calg$, we have
\[\Vert \pi(g)v^n_{s(g)}-v^n_{r(g)}\Vert \to 0 \quad \textrm{as}\ n\to \infty.\] 
We say that $\pi$ {\it contains the trivial representation weakly}, denoted by $1\prec \pi$, if there exists a sequence of normalized sections of $\sh$ almost invariant under $\pi(\calg)$.
\end{defn}


\subsection{Induced representations}

Let $\calg$ be a discrete measured groupoid on $(X, \mu)$ with $r, s\colon \calg \to X$ the range and source maps, respectively.
Let $\cals$ be a subgroupoid of $\calg$ and $\pi$ a representation of $\cals$ on a Hilbert bundle $\sh =\{ \sh_x\}_{x\in X}$ over $X$.
In this subsection, we construct a representation $\tilde{\pi}$ of $\calg$ on another Hilbert bundle $\tilde{\mathscr{H}} =\{ \tilde{\sh}_y\}_{y\in X}$ over $X$ canonically, as a generalization of usual induced representations of groups, and discuss its $c_0$-property and weak containment of the trivial representation.

For subsets $D$, $E$ of $\calg$ and an element $y\in X$, we set
\[DE =\{ \, gh\in \calg \mid g\in D,\ h\in E,\ s(g)=r(h)\, \} \quad \textrm{and}\quad D^y=D\cap \calg^y.\]
For $g\in \calg$, we denote $\{ g\} \cals$ by $g\cals$ simply.

We fix a measurable fundamental domain $D$ for right multiplication of $\cals$ on $\calg$.
That is, we fix a measurable subset $D$ of $\calg$ such that $D\cals =\calg$ and for any $y\in X$ and any distinct $g, h\in D^y$, the sets $g\cals$ and $h\cals$ are disjoint. 
For $y\in X$, we define $\tilde{\sh}_y$ as the set of maps $\xi \colon \calg^y\to \sh$ such that
\begin{itemize}
\item for any $g\in \calg^y$, we have $\xi(g)\in \sh_{s(g)}$;
\item for any $g\in \calg^y$ and any $h\in \cals^{s(g)}$, we have $\xi(gh)=\pi(h^{-1})\xi(g)$; and
\item we have $\sum_{g\in D^y}\Vert \xi(g)\Vert^2<\infty$.
\end{itemize}
By the second condition, any map $\xi \in \tilde{\sh}_y$ is determined by its restriction to $D^y$.
The space $\tilde{\sh}_y$ is a vector space by pointwise addition and scalar multiple, and is a separable Hilbert space by inner product $\langle \xi, \eta \rangle =\sum_{g\in D^y}\langle \xi(g), \eta(g)\rangle$ for $\xi, \eta \in \tilde{\sh}_y$.
We note that the space $\tilde{\sh}_y$ does not depend on the choice of the fundamental domain $D$.

\begin{lem}\label{lem-hb}
The family $\tilde{\sh}=\{ \tilde{\sh}_y\}_{y\in X}$ admits a structure of a Hilbert bundle over $X$.
\end{lem}

\begin{proof}
Let $(\sigma^n)_{n\in \N}$ be a fundamental sequence for $\sh$.
Pick a countable family $(\phi_m)_{m\in M}$ of measurable sections $\phi_m\colon X_m\to \calg$ of the range map $r$, defined on a measurable subset $X_m$ of $X$, such that $D =\bigsqcup_{m\in M}\phi_m(X_m)$ and for any $m\in M$, the map $s\circ \phi_m\colon X_m\to X$ is injective.

For $n\in \N$ and $m\in M$, we define a section $\tau^{n, m}$ of $\tilde{\sh}$.
Pick $y\in X$.
We first define a map $\tau^{n, m}_y\colon D^y\to \sh$ so that $\tau^{n, m}_y(g)\in \sh_{s(g)}$ for any $g\in D^y$.
The definition depends on whether $y\in X_m$ or not.
If $y\in X_m$, then let $\tau^{n, m}_y$ send $\phi_m(y)$ to $\sigma^n_{s\circ \phi_m(y)}$, and send any other point of $D^y$ to $0$.
If $y\not\in X_m$, then let $\tau^{n, m}_y$ send any point of $D^y$ to $0$. 
Extend $\tau^{n, m}_y$ to the map from $\calg^y$ to $\sh$ equivariant under right multiplication of $\cals$.
We get the element of $\tilde{\sh}_y$, denoted by the same symbol $\tau^{n, m}_y$.

The family $(\tau^{n, m})_{n, m}$ satisfies conditions (b) and (c) in Definition \ref{defn-hb}.
By Proposition \ref{prop-hb}, we have a unique standard Borel structure on $X\ast \tilde{\sh}$ such that $X\ast \tilde{\sh}$ is a Hilbert bundle over $X$ and $(\tau^{n, m})_{n, m}$ is its fundamental sequence.
\end{proof}

A section of $\tilde{\sh}$ is naturally regarded as a map from $\calg$ to $\sh$.
More precisely, if $\xi \colon X\to \tilde{\sh}$ is a section of $\tilde{\sh}$, then we have the map $\bar{\xi}\colon \calg \to \sh$ defined by $\bar{\xi}(g)=\xi_{r(g)}(g)$ for $g\in \calg$.
This map $\bar{\xi}$ satisfies that
\begin{itemize}
\item for any $g\in \calg$, we have $\bar{\xi}(g)\in \sh_{s(g)}$;
\item for any $g\in \calg$ and any $h\in \cals^{s(g)}$, we have $\bar{\xi}(gh)=\pi(h^{-1})\bar{\xi}(g)$; and
\item for any $y\in X$, we have $\sum_{g\in D^y}\Vert \bar{\xi}(g)\Vert^2<\infty$.
\end{itemize}
Conversely, any map $\bar{\xi}\colon \calg \to \sh$ satisfying these three conditions arises from a section of $\tilde{\sh}$.
The section $\xi$ is measurable if and only if $\bar{\xi}$ is measurable.
This holds true because, with the notation in the proof of Lemma \ref{lem-hb}, for any $n$, $m$ and any $y\in X$, we have
\[\langle \xi_y, \tau_y^{n, m}\rangle =\sum_{g\in D^y}\langle \xi_y(g), \tau_y^{n, m}(g)\rangle =\begin{cases}
\langle \bar{\xi}(\phi_m(y)), \sigma^n_{s\circ \phi_m(y)}\rangle & \textrm{if}\ y\in X_m\\
0 & \textrm{otherwise}.
\end{cases}
\]
We often identify $\xi$ with this map $\bar{\xi}$, denoting it by the same symbol $\xi$.

We define a homomorphism $\tilde{\pi}\colon \calg \to {\rm Iso}(\tilde{\sh})$ so that for $g\in \calg$, the Hilbert space isomorphism $\tilde{\pi}(g)\colon \tilde{\sh}_{s(g)}\to \tilde{\sh}_{r(g)}$ is given by
\[(\tilde{\pi}(g)\xi)(h)=\xi(g^{-1}h) \quad \textrm{for}\ \xi \in \tilde{\sh}_{s(g)}\ \textrm{and}\ h\in \calg^{r(g)}.\]
We got the representation $\tilde{\pi}$ of $\calg$ on the Hilbert bundle $\tilde{\sh}$, and call it the representation of $\calg$ {\it induced} from $\pi$.

\begin{prop}\label{prop-c0}
Let $\calg$ be a discrete measured groupoid on $(X, \mu)$ and $\cals$ a subgroupoid of $\calg$.
Let $\pi$ be a $c_0$-representation of $\cals$ on a Hilbert bundle $\sh$ over $X$.
Then the representation $\tilde{\pi}$ of $\calg$ induced from $\pi$ is also $c_0$.
\end{prop}

\begin{proof}
Let $\tilde{\sh}=\{ \tilde{\sh}_y\}_{y\in X}$ be the Hilbert bundle on which $\tilde{\pi}(\calg)$ acts.
Pick a fundamental sequence $(\sigma^n)_{n\in \N}$ for $\sh$.
For $n\in \N$, we define a section $\tilde{\sigma}^n$ of $\tilde{\sh}$ so that as a map from $\calg$ into $\sh$, we have $\tilde{\sigma}^n(g)=\pi(g^{-1})\sigma^n_{r(g)}$ for $g\in \cals$ and $\tilde{\sigma}^n(g)=0$ otherwise.

We check that the family $(\tilde{\sigma}^n)_{n\in \N}$ satisfies conditions (I) and (II) in Lemma \ref{lem-c0-check}.
If it is shown, then $\tilde{\pi}$ is $c_0$ by that lemma.
Let $D\subset \calg$ be a measurable fundamental domain for right multiplication of $\cals$ on $\calg$.
For any $y\in X$, the space $\tilde{\sh}_y$ is identified with the space of a map $\xi \colon D^y \to \sh$ such that $\xi(g)\in \sh_{s(g)}$ for any $g\in D^y$; and $\sum_{g\in D^y}\Vert \xi(g)\Vert^2<\infty$.
Fix $y\in X$.
For any $g\in D^y$, the element $\tilde{\pi}(g)\tilde{\sigma}^n_{s(g)}$ of $\tilde{\sh}_y$ is supported on the single point set $\{ g\}$ if it is regarded as a map from $D^y$.
We have $(\tilde{\pi}(g)\tilde{\sigma}^n_{s(g)})(g)=\tilde{\sigma}_{s(g)}^n(e_{s(g)})=\sigma_{s(g)}^n$, where for $x\in X$, we denote by $e_x$ the unit of $\calg$ at $x$.
For any $y\in X$, the set $\{ \, \tilde{\pi}(g)\tilde{\sigma}^n_{s(g)} \mid g\in D^y,\ n\in \N \, \}$ is thus total in $\tilde{\sh}_y$.
The family $(\tilde{\sigma}^n)_{n\in \N}$ satisfies condition (I) in Lemma \ref{lem-c0-check}.

To check condition (II), we fix $n, m\in \N$.
For any $g\in \calg \setminus \cals$, we have $\langle \tilde{\pi}(g)\tilde{\sigma}^n_{s(g)}, \tilde{\sigma}^m_{r(g)}\rangle =0$ because $\tilde{\sigma}^n$ and $\tilde{\sigma}^m$ are supported on $\cals$.
For any $g\in \cals$, denoting by $h$ the single element of $D^{r(g)}\cap \cals$, we have
\[\langle \tilde{\pi}(g)\tilde{\sigma}^n_{s(g)}, \tilde{\sigma}^m_{r(g)}\rangle=\langle \tilde{\sigma}^n(g^{-1}h), \tilde{\sigma}^m(h)\rangle =\langle \pi(g)\sigma^n_{s(g)}, \sigma^m_{r(g)}\rangle.\]
For any $\varepsilon >0$, the equation
\[\{ \, g\in \calg \mid |\langle \tilde{\pi}(g)\tilde{\sigma}^n_{s(g)}, \tilde{\sigma}^m_{r(g)}\rangle |>\varepsilon \, \} =\{ \, g\in \cals \mid |\langle \pi(g)\sigma^n_{s(g)}, \sigma^m_{r(g)}\rangle |>\varepsilon \, \}\]
therefore holds.
Condition (II) for $(\tilde{\sigma}^n)_{n\in \N}$ follows from that $\pi$ is $c_0$.
\end{proof}

\begin{prop}\label{prop-1}
Let $\calg$ be a discrete measured groupoid on $(X, \mu)$ and $\cals$ a subgroupoid of $\calg$.
Suppose that there exist an amenable discrete group $A$ and a homomorphism $\alpha \colon \calg \to A$ such that $\ker \alpha =\cals$ and either $A$ is cyclic or $\alpha$ is {\it class-surjective}, i.e., for a.e.\ $x\in X$, we have $\alpha(\calg_x)=A$.
Let $\pi$ be a representation of $\cals$ on a Hilbert bundle $\sh$ over $X$ with $1\prec \pi$.
We denote by $\tilde{\pi}$ the representation of $\calg$ induced from $\pi$.
Then $1\prec \tilde{\pi}$.
\end{prop}

To prove this proposition, we need an elementary fact on a F\o lner sequence.
Let $G$ be a discrete group.
Recall that a sequence $(F_n)_n$ of non-empty finite subsets of $G$ is called a {\it F\o lner sequence} of $G$ if for any $g\in G$, we have $|gF_n\bigtriangleup F_n|/|F_n|\to 0$ as $n\to \infty$.

\begin{lem}\label{lem-folner}
Let $A$ be a cyclic group.
Then there exists a F\o lner sequence $(F_n)_n$ of $A$ such that for any $n$, the set $F_n$ contains the neutral element of $A$; and for any non-empty subset $S$ of $A$ and for any $g\in A$, we have
\[\frac{\, |(gF_n\bigtriangleup F_n)\cap S|\, }{|F_n\cap S|}\to 0\quad \textrm{and}\quad \frac{\, |(gF_n\bigtriangleup F_n)\cap S|\, }{|gF_n\cap S|}\to 0\quad \textrm{as}\ n\to \infty.\]
\end{lem}

\begin{proof}
If $A$ is finite, then it is enough to set $F_n=A$ for any $n$.
We assume $A=\Z$.
For $n\in \N$, let $F_n$ be the set of integers whose absolute values are at most $n$.
Pick $m\in \Z$.
For any $n\in \N$, we have $|(m+F_n)\bigtriangleup F_n|=2|m|$.
Let $S$ be a non-empty subset of $\Z$.
If $S$ is finite, then the set $((m+F_n)\bigtriangleup F_n)\cap S$ is empty for any sufficiently large $n$.
If $S$ is infinite, then $|F_n\cap S|\to \infty$ and $|(m+F_n)\cap S|\to \infty$ as $n\to \infty$.
The lemma follows.
\end{proof}

\begin{proof}[Proof of Proposition \ref{prop-1}]
Let $\tilde{\sh}=\{ \tilde{\sh}_y\}_{y\in X}$ denote the Hilbert bundle on which $\tilde{\pi}(\calg)$ acts.
We first assume that $A$ is cyclic.
By Lemma \ref{lem-folner}, we have a F\o lner sequence $(F_n)_{n\in \N}$ of $A$ such that for any $n\in \N$, the set $F_n$ contains the neutral element of $A$; and for any non-empty subset $S$ and any $g\in A$, we have
\[\frac{\, |(gF_n\bigtriangleup F_n)\cap S|\, }{|F_n\cap S|}\to 0\quad \textrm{and}\quad \frac{\, |(gF_n\bigtriangleup F_n)\cap S|\, }{|gF_n\cap S|}\to 0\quad \textrm{as}\ n\to \infty.\]
By the assumption $1\prec \pi$, there exists a sequence $(v^m)_{m\in \N}$ of normalized sections of $\sh$ such that for a.e.\ $h\in \cals$, we have $\Vert \pi(h)v^m_{s(h)}-v^m_{r(h)}\Vert \to 0$ as $m\to \infty$.

Let $D\subset \calg$ be a measurable fundamental domain for right multiplication of $\cals$ on $\calg$.
For $n, m\in \N$, we define a map $\xi^{n, m}\colon D\to \sh$ so that for $g\in D$, the element $\xi^{n, m}(g)\in \sh_{s(g)}$ is given by
\[\xi^{n, m}(g)=\begin{cases}
v^m_{s(g)}/|\alpha(\calg^{r(g)})\cap F_n|^{1/2} & \textrm{if}\ \alpha(g)\in F_n\\
0 & \textrm{otherwise}.
\end{cases}\]
We extend this map to the map $\xi^{n, m}\colon \calg \to \sh$ so that for any $g\in D$ and any $h\in \cals^{s(g)}$, the equation $\xi^{n, m}(gh)=\pi(h^{-1})\xi^{n, m}(g)$ holds.
For a.e.\ $y\in X$, we have $\sum_{g\in D^y}\Vert \xi^{n, m}(g)\Vert^2=1$.
It follows that $\xi^{n, m}$ is regarded as a normalized section of $\tilde{\sh}$.

Pick an increasing sequence $(E_l)_{l\in \N}$ of measurable subsets of $\calg$ such that $\tilde{\mu}(E_l)<\infty$ for any $l\in \N$ and $\calg=\bigcup_{l\in \N}E_l$.

Fix $l\in \N$.
By the Lebesgue convergence theorem, there exists $n\in \N$ with
\begin{align}\label{folner}
\int_{E_l}\frac{\, |(\alpha(g)F_n\bigtriangleup F_n)\cap \alpha(\calg^{r(g)})|\, }{|F_n\cap \alpha(\calg^{r(g)})|}\, d\tilde{\mu}(g) & \leq \frac{1}{\, l^2\, }\quad \textrm{and}\\
\int_{E_l}\frac{\, |(\alpha(g)F_n\bigtriangleup F_n)\cap \alpha(\calg^{r(g)})|\, }{|\alpha(g)F_n\cap \alpha(\calg^{r(g)})|}\, d\tilde{\mu}(g) & \leq \frac{1}{\, l^2\, }.\label{folnerg}
\end{align}
For any $(g, g')\in \calg^{(2)}$, there exists a unique $h\in \cals^{s(g')}$ with $gg'h\in D$.
This element $h$ is denoted by $h(g, g')$.
By the Lebesgue convergence theorem, there exists $m\in \N$ with
\begin{equation}\label{h}
\int_{E_l}\sum_{g'\in D^{r(g)}\cap \, \alpha^{-1}(F_n)}|\langle \pi(h(g^{-1}, g'))v^m_{s(h(g^{-1}, g'))}, v^m_{s(g')}\rangle -1|\, d\tilde{\mu}(g)\leq \frac{1}{\, 2l^2\, }.
\end{equation}
We denote by $\eta^l$ the normalized section $\xi^{n, m}$ of $\tilde{\sh}$ with $n$ and $m$ chosen above.

Fix $g\in E_l$ and put $y=r(g)$ and $x=s(g)$.
For any $g'\in D^y$, putting $h=h(g^{-1}, g')$, we have
\begin{align*}
(\tilde{\pi}(g)\eta^l_x)(g') & =\eta^l_x(g^{-1}g')=\pi(h)\eta^l_x(g^{-1}g'h)\\
& =\begin{cases}
\pi(h)v^m_{s(h)}/|\alpha(\calg^x)\cap F_n|^{1/2} & \textrm{if}\ \alpha(g^{-1}g')\in F_n\\
0 & \textrm{otherwise},
\end{cases}
\end{align*}
and have
\[\eta^l_y(g')=\begin{cases}
v_{s(g')}^m/|\alpha(\calg^y)\cap F_n|^{1/2} & \textrm{if}\ \alpha(g')\in F_n\\
0 & \textrm{otherwise}.
\end{cases}\]
We set $A_1=\alpha(\calg^y)\cap \alpha(g)F_n=\alpha(g)(\alpha(\calg^x)\cap F_n)$ and $A_2=\alpha(\calg^y)\cap F_n$.
We also set
\[D_0=\{ \, g'\in D^y\mid \alpha(g')\in \alpha(g)F_n\cap F_n\, \}.\]
Setting $h(g')=h(g^{-1}, g')$ for $g'\in D^y$, we have
\begin{align*}
\langle \tilde{\pi}(g)\eta^l_x, \eta^l_y\rangle =\sum_{g'\in D^y}\langle (\tilde{\pi}(g)\eta^l_x)(g'), \eta^l_y(g')\rangle = \, \frac{1}{\, |A_1|^{1/2}|A_2|^{1/2}\, }\sum_{g'\in D_0}\langle \pi(h(g'))v^m_{s(h(g'))}, v^m_{s(g')}\rangle,
\end{align*}
and also have
\begin{align}\label{ineq-eta}
|\langle \tilde{\pi}(g)\eta^l_x, & \eta^l_y\rangle -1| \leq \left| 1- \frac{|D_0|}{\, |A_1|^{1/2}|A_2|^{1/2}\, }\right| + \sum_{g'\in D_0}\frac{\, |\langle \pi(h(g'))v^m_{s(h(g'))}, v^m_{s(g')}\rangle-1|\, }{|A_1|^{1/2}|A_2|^{1/2}}\\
& \leq \frac{1}{\, 4\, }\left( \frac{\, |A_1\bigtriangleup A_2|\, }{|A_1|}+\frac{\, |A_1\bigtriangleup A_2|\, }{|A_2|}\right) + \sum_{g'\in D_0}|\langle \pi(h(g'))v^m_{s(h(g'))}, v^m_{s(g')}\rangle-1|,\nonumber
\end{align}
where the last inequality holds because we have $|D_0|=|A_1\cap A_2|$ and have
\begin{align*}
0 & \, \leq 1-\frac{|A_1\cap A_2|}{\, |A_1|^{1/2}|A_2|^{1/2}\, } = 1 - \frac{\, |A_1|+|A_2|-|A_1\bigtriangleup A_2|\, }{2|A_1|^{1/2}|A_2|^{1/2}}\\
& \, \leq \frac{|A_1\bigtriangleup A_2|}{\, 2|A_1|^{1/2}|A_2|^{1/2}\, } \leq \frac{1}{\, 4\, }\left( \frac{\, |A_1\bigtriangleup A_2|\, }{|A_1|}+\frac{\, |A_1\bigtriangleup A_2|\, }{|A_2|}\right)
\end{align*}
by the inequality of arithmetic and geometric means.
Take integration of inequality (\ref{ineq-eta}) with respect to $g\in E_l$.
By inequalities (\ref{folner}), (\ref{folnerg}) and (\ref{h}), we obtain
\begin{equation}\label{ineq-int}
\int_{E_l}|\langle \tilde{\pi}(g)\eta^l_{s(g)}, \eta^l_{r(g)}\rangle -1|\, d\tilde{\mu}(g)\leq \frac{1}{\, l^2\, }.
\end{equation}

We therefore obtained the sequence $(\eta^l)_{l\in \N}$ of normalized sections of $\tilde{\sh}$ such that for any $k\in \N$, we have
\[\int_{E_k}|\langle \tilde{\pi}(g)\eta^l_{s(g)}, \eta^l_{r(g)}\rangle -1|\, d\tilde{\mu}(g)\leq \frac{1}{\, l^2\, }\]
for any $l\in \N$ with $l\geq k$.
It follows that for a.e.\ $g\in E_k$, we have $\langle \tilde{\pi}(g)\eta^l_{s(g)}, \eta^l_{r(g)}\rangle \to 1$ as $l\to \infty$ because the sum $\sum_{l=1}^\infty l^{-2}$ is convergent.
This pointwise convergence thus holds for a.e.\ $g\in \calg$.

We proved the proposition when $A$ is cyclic.
The proof of the case where $A$ is amenable and $\alpha$ is class-surjective is almost the same.
In fact, although we cannot use Lemma \ref{lem-folner}, for any F\o lner sequence $(F_n)_{n\in \N}$ of $A$, we can find $n\in \N$ satisfying inequalities (\ref{folner}) and (\ref{folnerg}) because $\alpha$ is class-surjective, i.e., we have $\alpha(\calg^{r(g)})=A$ for a.e.\ $g\in \calg$.
\end{proof}


\subsection{The Haagerup property and its generalization}

The Haagerup property (HAP) of groups was initially discovered by Haagerup \cite{h}.
On the basis of this work, the HAP was introduced for p.m.p.\ discrete measured equivalence relations by Jolissaint \cite{jol-h}, and was introduced for discrete measured groupoids by Anantharaman-Delaroche \cite{ana-h}.
For our purpose, we introduce a property of discrete measured groupoids similar to the HAPs in \cite{jol-h} and \cite{ana-h}, and call it the HAP in the present paper.
Comparison among these HAPs is discussed in Appendix \ref{sec-app}.

Let $\calg$ be a discrete measured groupoid on $(X, \mu)$.
A measurable function $\phi \colon \calg \to \C$ is called {\it positive definite} if for a.e.\ $x\in X$, any $g_1,\ldots, g_n\in \calg^x$ and any $\lambda_1,\ldots, \lambda_n\in \C$, we have $\sum_{i, j=1}^n\overline{\lambda}_i\lambda_j\phi(g_i^{-1}g_j)\geq 0$.
A positive definite function $\phi$ on $\calg$ is called {\it normalized} if $\phi(e_x)=1$ for a.e.\ $x\in X$, where $e_x$ is the unit of $\calg$ at $x$.
A measurable function $\phi$ on $\calg$ is called $c_0$ (or a {\it $c_0$-function}) if for any $\delta >0$ and any $\varepsilon >0$, there exists a measurable subset $Y$ of $X$ such that $\mu(X\setminus Y)<\delta$ and for a.e.\ $x\in Y$, we have
\[|\{ \, g\in (\calg|_Y)_x \mid |\phi(g)|>\varepsilon \, \}|<\infty.\]

Let $\pi$ be a representation of $\calg$ on a Hilbert bundle $\sh$ over $X$ and $v$ a section of $\sh$.
The function $g\mapsto \langle \pi(g)v_{s(g)}, v_{r(g)}\rangle$ on $\calg$ is then positive definite.
If $v$ is normalized, then the function is normalized.
If $\pi$ is $c_0$, then the function is $c_0$.

\begin{rem}\label{rem-pdf-basic}
Let $\phi \colon \calg \to \C$ be a positive definite function.
Then
\begin{enumerate}
\item[(i)] for a.e.\ $x\in X$, we have $\phi(e_x)\geq 0$; and
\item[(ii)] for a.e.\ $g\in \calg$, we have $\phi(g^{-1})=\overline{\phi(g)}$ and $|\phi(g)|^2\leq \phi(e_{r(g)})\phi(e_{s(g)})$.
\end{enumerate}
In fact, assertion (i) follows by definition.
Assertion (ii) holds because for a.e.\ $g\in \calg$, the matrix
\[\left(\begin{matrix}
\phi(e_{r(g)}) & \phi(g)\\
\phi(g^{-1}) & \phi(e_{s(g)})
\end{matrix}\right)\]
is non-negative.
\end{rem}

The following is a basic property of positive definite functions:

\begin{lem}[\ci{Lemma 13}{ana-h}]\label{lem-rest-pdf}
Let $\calg$ be a discrete measured groupoid on $(X, \mu)$ and $\cals$ a subgroupoid of $\calg$.
Let $\phi$ be a positive definite function on $\cals$.
Then the function $\psi \colon \calg \to \C$ defined by $\psi =\phi$ on $\cals$ and $\psi =0$ on $\calg \setminus \cals$ is positive definite on $\calg$.
\end{lem}

\begin{defn}\label{defn-hap}
We say that a discrete measured groupoid $\calg$ on $(X, \mu)$ has the {\it Haagerup property (HAP)} if there exists a sequence $(\phi_n)_n$ of normalized positive definite $c_0$-functions on $\calg$ such that for a.e.\ $g\in \calg$, we have $\phi_n(g)\to 1$ as $n\to \infty$.
\end{defn}

\begin{rem}
The HAP in Definition \ref{defn-hap} is apparently quite similar, but different from the HAPs of Jolissaint \cite{jol-h} and Anantharaman-Delaroche \cite{ana-h} (see Appendix \ref{sec-app}).
The difference is the $c_0$-property of positive definite functions. 
Our definition motivates from Proposition \ref{prop-hap-ext}, whose proof is based on Lemmas \ref{lem-c0-pre-check} and \ref{lem-c0-check}.
These two lemmas provide a sufficient condition for a representation to be $c_0$.
It does not seem obvious whether the same criterion is available for the $c_0$-property in the definitions of their HAPs (see Definitions \ref{defn-hap-jol} and \ref{defn-hap-ad}).
Nevertheless, in Corollary \ref{cor-hap}, we show that all of Jolissaint's HAP, Anantharaman-Delaroche's HAP and our HAP are equivalent.
\end{rem}

\begin{rem}\label{rem-pdf}
A measurable function $\phi \colon \calg \to \C$ is $c_0$ if and only if for any $\delta >0$, there exists a measurable subset $Y$ of $X$ such that $\mu(X\setminus Y)<\delta$ and for any $\varepsilon >0$, we have
\[\tilde{\mu}(\{ \, g\in \calg|_Y \mid |\phi(g)|>\varepsilon \, \})<\infty,\]
that is, the restriction of $\phi$ to $\calg|_Y$ satisfies the $c_0$-property in Definitions \ref{defn-hap-jol} and \ref{defn-hap-ad} of the HAPs of Jolissaint and Anantharaman-Delaroche.
\end{rem}

To a normalized positive definite function $\phi$ on $\calg$, the GNS representation $\pi^\phi$ of $\calg$ is associated as follows:
For $x\in X$, we have inner product $\langle \cdot, \cdot \rangle$ on the vector space whose basis is $\calg^x$, defined by $\langle g, h\rangle =\phi(h^{-1}g)$ for $g, h\in \calg^x$.
Its completion is denoted by $\sh^\phi_x$.
The family $\sh^\phi =\{ \sh^\phi_x\}_{x\in X}$ admits a structure of a Hilbert bundle over $X$.
In fact, if $(\sigma_n)_n$ is a sequence of measurable sections of the range map of $\calg$ with $\calg =\bigcup_n\sigma_n(X)$, then it is a fundamental sequence of $\sh^\phi$.
We have the homomorphism $\pi^\phi \colon \calg \to {\rm Iso}(\sh^\phi)$ defined by $\pi^\phi(g)h=gh$ for $g\in \calg$ and $h\in \calg^{s(g)}$.
The representation $\pi^\phi$ is called the {\it GNS representation} of $\calg$ associated to $\phi$.
Let $v^\phi$ be the section of $\sh^\phi$ defined so that for $x\in X$, $v^\phi_x$ is the unit of $\calg$ at $x$.
For any $g\in \calg$, we then have $\phi(g)=\langle \pi^\phi(g)v^\phi_{s(g)}, v^\phi_{r(g)}\rangle$.
By Lemma \ref{lem-c0-check}, if $\phi$ is $c_0$, then $\pi^\phi$ is $c_0$.

The following lemma is deduced from the relationship, observed above, between positive definite functions on $\calg$ and matrix coefficients of a representation of $\calg$.

\begin{lem}\label{lem-hap-rep}
Let $\calg$ be a discrete measured groupoid on $(X, \mu)$.
Suppose that $\calg$ has the HAP, that is, there exists a sequence $(\phi_n)_{n\in \N}$ of normalized positive definite $c_0$-functions on $\calg$ such that for a.e.\ $g\in \calg$, we have $\phi_n(g)\to 1$ as $n\to \infty$.
Let $\pi$ denote the direct sum of the GNS representations $\pi^{\phi_n}$ with $n\in \N$.
Then $\pi$ is $c_0$ and $1\prec \pi$.

Conversely, if there exists a $c_0$-representation $\tau$ of $\calg$ on a Hilbert bundle over $X$ with $1\prec \tau$, then $\calg$ has the HAP.
\end{lem}

\begin{prop}\label{prop-hap-ext}
Let $\calg$ be a discrete measured groupoid on $(X, \mu)$ and $\cals$ a subgroupoid of $\calg$.
Suppose that there exist an amenable discrete group $A$ and a homomorphism $\alpha \colon \calg \to A$ such that $\ker \alpha =\cals$ and either $A$ is abelian or $\alpha$ is class-surjective.
If $\cals$ has the HAP, then so does $\calg$.
\end{prop}

\begin{proof}
Since $\cals$ has the HAP, by Lemma \ref{lem-hap-rep}, there exists a $c_0$-representation $\pi$ of $\cals$ on a Hilbert bundle over $X$ with $1\prec \pi$.
Let $\tilde{\pi}$ denote the representation of $\calg$ induced from $\pi$.
If either $A$ is cyclic or $\alpha$ is class-surjective, then by Propositions \ref{prop-c0} and \ref{prop-1}, $\tilde{\pi}$ is $c_0$, and we have $1\prec \tilde{\pi}$.
It follows from Lemma \ref{lem-hap-rep} that $\calg$ has the HAP.
If $A$ is not necessarily cyclic and is abelian, then by inductive argument, for any finitely generated subgroup $B$ of $A$, the subgroupoid $\alpha^{-1}(B)$ of $\calg$ has the HAP.
By Lemma \ref{lem-rest-pdf}, this implies that $\calg$ has the HAP.
\end{proof}

The following generalizes \cite[Corollary 2.8]{jol-h}:

\begin{prop}\label{prop-hap-rest}
Let $\calg$ be a discrete measured groupoid on a standard probability space $(X, \mu)$.
Let $A$ be a measurable subset of $X$ with $\calg A=X$.
Then $\calg|_A$ has the HAP if and only if so does $\calg$.
\end{prop}

\begin{proof}
The ``if" part follows by definition.
We prove the ``only if" part.
Suppose that $\calg|_A$ has the HAP.
There exists a sequence $(\phi_n)_{n\in \N}$ of normalized positive definite $c_0$-functions on $\calg|_A$ such that for a.e.\ $g\in \calg|_A$, we have $\phi_n(g)\to 1$ as $n\to \infty$.

Put $Y_0=A$.
We define a map $f_0\colon Y_0\to \calg$ by $f_0(x)=e_x$ for $x\in Y_0$.
By the assumption $\calg A=X$, there exists a family $(f_k)_{k\in \N}$ of countably many measurable sections $f_k\colon Y_k\to \calg$ of the source map $s$, defined on a measurable subset $Y_k$ of $X\setminus Y_0$, such that $\bigsqcup_{k\in \N}Y_k=X\setminus Y_0$ and for any $k\in \N$, the map $r\circ f_k$ is an injection from $Y_k$ into $Y_0$.
For $n\in \N$, we define a function $\psi_n\colon \calg \to \C$ so that for a.e.\ $g\in \calg$, choosing $k, l\in \{ 0\} \cup \N$ with $r(g)\in Y_k$ and $s(g)\in Y_l$, we have
\[\psi_n(g)=\phi_n(f_k(r(g))gf_l(s(g))^{-1}).\]
The function $\psi_n$ is positive definite and normalized.
For a.e.\ $g\in \calg$, we have $\psi_n(g)\to 1$ as $n\to \infty$.

We fix $n\in \N$ and keep the notation in the last paragraph.
We claim that $\psi_n$ is $c_0$.
To prove it, we pick $\delta >0$ and $\varepsilon >0$.
We have to find a measurable subset $Z$ of $X$ such that $\mu(X\setminus Z)<\delta$ and for a.e.\ $x\in Z$, we have
\[|\{ \, g\in (\calg|_Z)_x \mid |\psi_n(g)|>\varepsilon \, \}|<\infty.\]
There exists $K\in \N$ with $\mu(X\setminus (\bigcup_{k=0}^KY_k))<\delta/2$.
Since $\phi_n$ is a $c_0$-function on $\calg|_{Y_0}$, there exists a measurable subset $W$ of $Y_0$ such that for any integer $k$ with $0\leq k\leq K$, putting $Z_k=(r\circ f_k)^{-1}(W)$, we have $\mu(Y_k\setminus Z_k)<\delta/(2(K+1))$; and for a.e.\ $y\in W$, we have
\[|\{ \, h\in (\calg|_W)_y\mid |\phi_n(h)|>\varepsilon \, \}|<\infty.\]
We set $Z=\bigcup_{k=0}^KZ_k$.
We have $\mu(X\setminus Z)<\delta$.

Fix $x\in Z$.
There exists a unique integer $k$ with $0\leq k\leq K$ and $x\in Z_k$.
We have
\[\{ \, g\in (\calg|_Z)_x\mid |\psi_n(g)|>\varepsilon \, \} =\bigsqcup_{l=0}^K G_l,\]
where for an integer $l$ with $0\leq l\leq K$, we set
\[G_l=\{ \, g\in (\calg|_Z)_x\mid r(g)\in Z_l,\ |\psi_n(g)|>\varepsilon \, \}.\]
The map sending an element $g$ of $G_l$ to the element $f_l(r(g))gf_k(x)^{-1}$ is an injection from $G_l$ into the set
\[\{ \, h\in (\calg|_W)_{r\circ f_k(x)}\mid |\phi_n(h)|>\varepsilon \, \}.\]
The last set is finite for a.e.\ $x\in Z_k$.
The claim follows.

We proved that the sequence $(\psi_n)_n$ satisfies the condition in Definition \ref{defn-hap}.
It follows that $\calg$ has the HAP.
\end{proof}

We now introduce the generalized Haagerup property for a discrete measured groupoid, following \cite[Definition 4.2.1]{ccjjv}.
For an abelian discrete group $C$, we denote by $\widehat{C}$ the group of characters on $C$.
Let $1_C$ denote the trivial character on $C$.
We mean by a compact neighborhood of $1_C$ in $\widehat{C}$ the compact subset of $\widehat{C}$ containing $1_C$ in its interior.

\begin{defn}\label{defn-ghap}
Let $\calg$ be a discrete measured groupoid on $(X, \mu)$ and $(C, \iota)$ a central subgroupoid of $\calg$.
We say that the pair $(\calg, (C, \iota))$ has the {\it generalized Haagerup property (gHAP)} if there exist a sequence $(V_n)_{n\in \N}$ of compact neighborhoods of $1_C$ in $\widehat{C}$ and a family $\{ \, \phi^n_\chi \mid n\in \N,\ \chi \in V_n\, \}$ of normalized positive definite functions on $\calg$ satisfying the following conditions (1)--(4):
\begin{enumerate}
\item[(1)] For a.e.\ $g\in \calg$ and any $\varepsilon >0$, there exists $N\in \N$ such that for any $n\geq N$ and any $\chi \in V_n$, we have $|\phi^n_\chi(g)-1|<\varepsilon$.
\item[(2)] For any $n\in \N$, any $\chi \in V_n$, a.e.\ $g\in \calg$ and any $c\in C$, we have $\phi_\chi^n(gc)=\phi_\chi^n(g)\chi(c)$.
\item[(3)] For any $n\in \N$ and a.e.\ $g\in \calg$, the function $\chi \mapsto \phi_\chi^n(g)$ on $V_n$ is continuous.
\item[(4)] For any $n\in \N$ and any $\chi \in V_n$, the function on $\calg/C$ induced by $|\phi_\chi^n|$ is $c_0$.
\end{enumerate}
\end{defn}

Let us refer to condition (1) as the condition that the family $\{ \phi^n_\chi \}_{n, \chi}$ \textit{approaches} $1$ on $\calg$ as $n\to \infty$.
For any $n\in \N$, taking a countable dense subset of $V_n$ and using condition (3), we see that the function $g\mapsto \sup_{\chi \in V_n}|\phi_\chi^n(g)-1|$ on $\calg$ is measurable.
Condition (1) is then equivalent to that this function converges to 0 pointwise almost everywhere as $n\to \infty$.

If $X$ consists of a single point and $\calg$ is a discrete group, then the gHAP in Definition \ref{defn-ghap} is equivalent to that in \cite[Definition 4.2.1]{ccjjv}.

The following is similar to Proposition \ref{prop-hap-ext}.

\begin{prop}\label{prop-ghap-ext}
Let $\calg$ be a discrete measured groupoid on $(X, \mu)$ and $\cals$ a subgroupoid of $\calg$.
Let $(C, \iota)$ be a central subgroupoid of $\calg$ such that the image of $\iota$ is contained in $\cals$.
Suppose that there exist an amenable discrete group $A$ and a homomorphism $\alpha \colon \calg \to A$ such that $\ker \alpha =\cals$ and either $A$ is abelian or $\alpha$ is class-surjective.
If the pair $(\cals, (C, \iota))$ has the gHAP, then so does the pair $(\calg, (C, \iota))$.
\end{prop}

\begin{proof}
By assumption, there exist a sequence $(V_n)_{n\in \N}$ and a family $\{ \, \phi^n_\chi \mid n\in \N,\ \chi \in V_n\, \}$ satisfying conditions (1)--(4) in Definition \ref{defn-ghap} for $\cals$ in place of $\calg$.
We first assume that $A$ is abelian.
We show the following assertion:
For any cyclic subgroup $B$ of $A$, there exist a subsequence $(V_{n_l})_l$ of $(V_n)_n$ and a family $\{ \, \psi^l_\chi \mid l\in \N,\ \chi \in V_{n_l}\, \}$ of normalized positive definite functions on $\calg$ such that the family $\{ \psi^l_\chi \}_{l, \chi}$ approaches $1$ on $\alpha^{-1}(B)$ as $l\to \infty$ and satisfies conditions (2)--(4) for $\calg$.
If the assertion is shown, then by inductive argument, the same assertion holds true for any finitely generated subgroup of $A$ in place of $B$.
This implies the proposition by the diagonal argument.

Let $B$ be a cyclic subgroup of $A$.
We put $\cal{E}=\alpha^{-1}(B)$.
To lighten symbols, for $m\in \N$, we set $Z_m=\{ m\} \times V_m$ and set $Z=\bigcup_{m\in \N}Z_m$.
For each $z=(m, \chi)\in Z$, let $\pi^z$ denote the GNS representation of $\cals$ associated to $\phi_\chi^m$ and $\sh^z$ denote the Hilbert bundle over $X$ on which $\pi^z(\cals)$ acts.
We have the normalized section $v^z$ of $\sh^z$ with $\phi^m_\chi(g)=\langle \pi^z(g)v^z_{s(g)}, v^z_{r(g)}\rangle$ for a.e.\ $g\in \cals$.
Let $\tilde{\pi}^z$ denote the representation of $\calg$ induced from $\pi^z$ and $\tilde{\sh}^z$ denote the Hilbert bundle over $X$ on which $\tilde{\pi}^z(\calg)$ acts.

Recall the construction of an almost invariant sequence in the proof of Proposition \ref{prop-1}.
Let $D\subset \calg$ be a measurable fundamental domain for right multiplication of $\cals$ on $\calg$.
By Lemma \ref{lem-folner}, there exists a F\o lner sequence $(F_n)_{n\in \N}$ of $B$ such that for any $n\in \N$, the set $F_n$ contains the neutral element of $B$; and for any non-empty subset $S$ of $B$ and any $g\in B$, we have
\[\frac{\, |(gF_n\bigtriangleup F_n)\cap S|\, }{|F_n\cap S|}\to 0\quad \textrm{and}\quad \frac{\, |(gF_n\bigtriangleup F_n)\cap S|\, }{|gF_n\cap S|}\to 0\quad \textrm{as}\ n\to \infty.\]
For $n\in \N$ and $z\in Z$, we define a map $\xi^{n, z}\colon D\to \sh^z$ so that for $g\in D$, the element $\xi^{n, z}(g)$ of $\sh_{s(g)}^z$ is given by
\[\xi^{n, z}(g)=\begin{cases}
v^z_{s(g)}/|\alpha(\calg^{r(g)})\cap F_n|^{1/2} & \textrm{if}\ \alpha(g)\in F_n\\
0 & \textrm{otherwise}.
\end{cases}\]
We extend this map to the map $\xi^{n, z}\colon \calg \to \sh^z$ so that for any $g\in D$ and any $h\in \cals^{s(g)}$, the equation $\xi^{n, z}(gh)=\pi^z(h^{-1})\xi^{n, z}(g)$ holds.
The map $\xi^{n, z}$ is then regarded as a normalized section of $\tilde{\sh}^z$.

Pick an increasing sequence $(E_l)_{l\in \N}$ of measurable subsets of $\cal{E}$ such that $\tilde{\mu}(E_l)<\infty$ for any $l\in \N$ and $\cal{E}=\bigcup_{l\in \N}E_l$.

Fix $l\in \N$.
By the Lebesgue convergence theorem, there exists $n\in \N$ with
\begin{align}\label{folner2nd}
\int_{E_l}\frac{\, |(\alpha(g)F_n\bigtriangleup F_n)\cap \alpha(\calg^{r(g)})|\, }{|F_n\cap \alpha(\calg^{r(g)})|}\, d\tilde{\mu}(g) & \leq \frac{1}{\, l^2\, }\quad \textrm{and}\\
\int_{E_l}\frac{\, |(\alpha(g)F_n\bigtriangleup F_n)\cap \alpha(\calg^{r(g)})|\, }{|\alpha(g)F_n\cap \alpha(\calg^{r(g)})|}\, d\tilde{\mu}(g) & \leq \frac{1}{\, l^2\, }.\label{folnerg2nd}
\end{align}
For any $(g, g')\in \calg^{(2)}$, there exists a unique $h\in \cals^{s(g')}$ with $gg'h\in D$.
This element $h$ is denoted by $h(g, g')$.
For any $m\in \N$, the function $g\mapsto \sup_{\chi \in V_m}|\phi_\chi^m(g)-1|$ on $\cals$ is measurable.
This function converges to 0 pointwise almost everywhere as $m\to \infty$ because the family $\{ \phi^m_\chi \}_{m, \chi}$ approaches $1$ on $\cals$ as $m\to \infty$.
There thus exists $m\in \N$ such that
\begin{equation*}
\int_{E_l}\sum_{g'\in D^{r(g)}\cap \, \alpha^{-1}(F_n)}\left( \sup_{\chi \in V_m}|\phi_\chi^m(h(g^{-1}, g')) -1| \right) d\tilde{\mu}(g)\leq \frac{1}{\, 2l^2\, }.
\end{equation*}
We set $U_l=V_m$ and $W_l= \{ l\} \times U_l$.
For $w=(l, \chi)\in W_l$, we set $z(w)=(m, \chi)\in Z_m$ and denote by $\eta^w$ the normalized section $\xi^{n, z(w)}$ of $\tilde{\sh}^{z(w)}$.
Following the computation in the proof of Proposition \ref{prop-1} to deduce inequality (\ref{ineq-int}), we have
\begin{equation}\label{psi-l}
\int_{E_l}\left( \sup_{w\in W_l}|\langle \tilde{\pi}^{z(w)}(g)\eta^w_{s(g)}, \eta^w_{r(g)}\rangle -1|\right) d\tilde{\mu}(g)\leq \frac{1}{\, l^2\, }.
\end{equation}

We set $W=\bigcup_{l\in \N}W_l$.
For $w=(l, \chi)\in W$, we define a function $\psi_\chi^l \colon \calg \to \C$ by 
\[\psi_\chi^l(g)=\langle \tilde{\pi}^{z(w)}(g)\eta^w_{s(g)}, \eta^w_{r(g)}\rangle \quad \textrm{for}\ g\in \calg.\]
By inequality (\ref{psi-l}), the family $\{ \psi_\chi^l\}_{l, \chi}$ approaches $1$ on $\cal{E}$ as $l\to \infty$.
We check conditions (2)--(4) in Definition \ref{defn-ghap} for this family and the sequence $(U_l)_l$.

Fix $l\in \N$ and $g\in \calg$.
Put $y=r(g)$ and $x=s(g)$.
For any $\chi \in U_l$, putting $w=(l, \chi)$, $z=z(w)=(m, \chi)$ and $h(g')=h(g^{-1}, g')$ for $g'\in D^y$, we have
\begin{align}\label{psi-phi}
\psi_\chi^l(g) & =\langle \tilde{\pi}^z(g)\eta^w_x, \eta^w_y\rangle =\sum_{g'\in D^y}\langle \eta_x^w(g^{-1}g'), \eta^w_y(g')\rangle \\
& =\sum_{g'\in D^y}\langle \pi^z(h(g'))\eta_x^w(g^{-1}g'h(g')), \eta^w_y(g')\rangle  \nonumber \\
& =|\alpha(\calg^x)\cap F_n|^{-1/2}|\alpha(\calg^y)\cap F_n|^{-1/2}\sum_{g'\in D_0}\phi^m_\chi(h(g'))\nonumber
\end{align}
where $n$ is the number chosen in inequalities (\ref{folner2nd}) and (\ref{folnerg2nd}) for $l$ and
\[D_0=\{ \, g'\in D^y\mid \alpha(g')\in \alpha(g)F_n\cap F_n\, \}.\]
For any $c\in C$ and any $g'\in D^y$, we have $h((gc)^{-1}, g')=h(g^{-1}, g')c$ by the definition of the symbol $h(\cdot, \cdot)$.
Equation (\ref{psi-phi}) implies that for a.e.\ $g\in \calg$, we have $\psi_\chi^l(gc)=\chi(c)\psi_\chi^l(g)$.
Equation (\ref{psi-phi}) also implies that for a.e.\ $g\in \calg$, the function $\chi \mapsto \psi_\chi^l(g)$ on $U_l$ is continuous.

We have checked conditions (2) and (3) for the families $(U_l)_l$ and $\{ \psi_\chi^l\}_{l, \chi}$.
Condition (4) for them is checked by following the proof of Proposition \ref{prop-c0}.
Since $\phi_\chi^n$ is not necessarily $c_0$ as a function on $\cals$, we have to modify the proof and have to show a lemma corresponding to Lemma \ref{lem-c0-check}.
This process is established almost verbatim.
We thus omit it.

As in the proof of Proposition \ref{prop-1}, the proof of the case where $\alpha$ is class-surjective is obtained similarly.
\end{proof}

Along the proof of Proposition \ref{prop-hap-rest}, we can show the following:

\begin{prop}\label{prop-ghap-rest}
Let $\calg$ be a discrete measured groupoid on a standard probability space $(X, \mu)$, and $(C, \iota)$ a central subgroupoid of $\calg$.
Let $A$ be a measurable subset of $X$ with $\calg A=X$.
We use the same symbol $\iota$ to denote the restriction of $\iota$ to $(C\ltimes (X, \mu))|_A$
Then the pair $(\calg|_A, (C, \iota))$ has the gHAP if and only if so does the pair $(\calg, (C, \iota))$.
\end{prop}


\subsection{Comparison with the gHAP of groups}

When a discrete measured groupoid $\calg$ is associated with a p.m.p.\ action $G\c (X, \mu)$, it is natural to ask the relationship between the HAPs of $\calg$ and $G$.
Jolissaint \cite[Remark in p.172]{jol-h} shows that under the assumption that the action is free, $\calg$ has his HAP if and only if $G$ has the HAP, through \cite[Proposition 3.1]{popa}.
We establish this equivalence for our HAP and gHAP in the following:

\begin{prop}\label{prop-group-oid}
Let $G$ be a discrete group having a central subgroup $C$.
Suppose that $G/C$ has a p.m.p.\ action on $(X, \mu)$, and let $G$ act on $(X, \mu)$ through the quotient map from $G$ onto $G/C$.
We set $\calg =G\ltimes (X, \mu)$ and $\cal{C}=C\ltimes (X, \mu)$, and define $\iota$ as the identity on $\cal{C}$.
Then the pair $(\calg, (C, \iota))$ has the gHAP if and only if the pair $(G, C)$ has the gHAP.
\end{prop}

Assuming that $C$ is trivial, we see that $\calg$ has the HAP if and only if $G$ has the HAP.

\begin{proof}[Proof of Proposition \ref{prop-group-oid}]
For a function $\psi \colon G\to \C$, we define a function $\tilde{\psi}\colon \calg \to \C$ by $\tilde{\psi}(\gamma, x)=\psi(\gamma)$ for $\gamma \in G$ and $x\in X$.
If $\psi$ is positive definite, then so is $\tilde{\psi}$.
The ``if" part of the proposition follows from this construction.

We prove the ``only if" part of the proposition.
For a measurable function $\phi \colon \calg \to \C$ with $|\phi(g)|\leq 1$ for a.e.\ $g\in \calg$, we define a function $\bar{\phi}\colon G\to \C$ by $\bar{\phi}(\gamma)=\int_X\phi(\gamma, x)\, d\mu(x)$.
For any $\gamma, \delta \in G$, we have
\[\bar{\phi}(\gamma^{-1}\delta)=\int_X\phi(\gamma^{-1}\delta, \delta^{-1}x)\, d\mu(x)=\int_X\phi((\gamma, \gamma^{-1}x)^{-1}(\delta, \delta^{-1}x))\, d\mu(x),\]
where the first equation holds because the action $G\c (X, \mu)$ is p.m.p.
It follows that if $\phi$ is positive definite and normalized, then so is $\bar{\phi}$.

Suppose that we have a sequence $(V_n)_{n\in \N}$ of compact neighborhoods of $1_C$ in $\widehat{C}$ and a family $\{ \, \phi^n_\chi \mid n\in \N,\ \chi \in V_n\, \}$ of normalized positive definite functions on $\calg$ satisfying conditions (1)--(4) in Definition \ref{defn-ghap}.
We check those conditions (1)--(4) for the pair $(G, C)$ and the families $(V_n)_n$ and $\{ \bar{\phi}^n_\chi \}_{n, \chi}$.
Conditions (2) and (3) hold by the definition of $\bar{\phi}^n_\chi$.

As mentioned right after Definition \ref{defn-ghap}, for any $n\in \N$, the function on $\calg$ defined by $g\mapsto \sup_{\chi \in V_n}|\phi_\chi^n(g)-1|$ is measurable.
Condition (1) for $\{ \phi_\chi^n\}_{n, \chi}$ implies that this function converges to 0 pointwise almost everywhere as $n\to \infty$.
The Lebesgue convergence theorem implies condition (1) for $\{ \bar{\phi}^n_\chi \}_{n, \chi}$.

To check condition (4), we fix $n\in \N$ and $\chi \in V_n$, and put $\phi=\phi^n_\chi$.
Pick $\varepsilon >0$.
We have to show that the set $\{ \, \gamma \in G/C\mid |\bar{\phi}(\gamma)|>\varepsilon \, \}$ is finite.
Since $|\phi|$ is $c_0$ on $\calg /C$, there exist a measurable subset $Y$ of $X$ and a finite subset $F$ of $G/C$ such that $\mu(X\setminus Y)<\varepsilon /3$; and for a.e.\ $x\in Y$ and any $\gamma \in (G/C)\setminus F$ with $\gamma x\in Y$, we have $|\phi(\gamma, x)|\leq \varepsilon /3$.
It follows that for any $\gamma \in (G/C)\setminus F$, we have
\[|\bar{\phi}(\gamma)|\leq \int_{Y\, \cap \, \gamma^{-1}Y}|\phi(\gamma, x)|\, d\mu(x)+\mu(X\setminus Y)+\mu(X\setminus \gamma^{-1}Y)<\varepsilon,\]
where in the last inequality, we use the assumption that the action $G/C\c (X, \mu)$ is p.m.p.
Condition (4) follows.
The pair $(G, C)$ thus has the gHAP.
\end{proof}


\section{The Haagerup property of groups in class $\scc$}\label{sec-c}

Recall that $\scc$ denotes the smallest subclass of the class of discrete groups that satisfies the following conditions (1)--(4):
\begin{enumerate}
\item[(1)] Any treeable group belongs to $\scc$.
\item[(2)] The direct product of two groups in $\scc$ belongs to $\scc$.
\item[(3)] For a discrete group $G$ and its finite index subgroup $H$, we have $G\in \scc$ if and only if $H\in \scc$.
\item[(4)] Any central extension of a group in $\scc$ belongs to $\scc$.
\end{enumerate}
To show that any group in $\scc$ has the HAP, we introduce subclasses of $\scc$.
We define $\sd$ as the smallest subclass of $\scc$ satisfying conditions (1), (2) and (4).
We inductively define subclasses of $\sd$, $\sd_n$ and $\sd_n'$ for $n\in \N$, as follows:
Let $\sd_1$ be the class of direct products of finitely many treeable groups, and $\sd_1'$ the class of central extensions of a group in $\sd_1$.
Let $\sd_n$ be the class of direct products of finitely many groups in $\sd_{n-1}'$.
Let $\sd_n'$ be the class of central extensions of a group in $\sd_n$.
We have the inclusion $\sd_1\subset \sd_1'\subset \sd_2\subset \sd_2'\subset \cdots$.
The union $\bigcup_n\sd_n$ satisfies conditions (1), (2) and (4), and is therefore equal to $\sd$.

\begin{lem}\label{lem-d}
Pick $n\in \N$ and a group $G$ in $\sd_n'$.
Let $H$ be a finite index subgroup of $G$.
Then there exists a finite index subgroup $H_1$ of $H$ with $H_1\in \sd_n'$.
\end{lem}

\begin{proof}
We prove the lemma by induction on $n$.
The case of $n=1$ follows from that any finite index subgroup of a treeable group is treeable (\cite[Th\'eor\`eme IV.4]{gab-cost}).
Suppose that $G$ is in $\sd_n'$ and $H$ is a finite index subgroup of $G$.
There exists a central subgroup $C$ of $G$ with $G/C\in \sd_n$.
Set $\Gamma =G/C$.
We have $\Gamma_1,\ldots, \Gamma_m\in \sd_{n-1}'$ with $\Gamma =\Gamma_1\times \cdots \times \Gamma_m$.
Let $q\colon G\to \Gamma$ be the quotient map.
By the hypothesis of the induction, for any $i=1,\ldots, m$, there exists a finite index subgroup $\Lambda_i$ of $\Gamma_i$ such that $\Lambda_i\in \sd_{n-1}'$, and setting $\Lambda =\Lambda_1\times \cdots \times \Lambda_m$, we have $\Lambda <q(H)$.
We set $H_1=q^{-1}(\Lambda)\cap H$.
The group $H_1$ is of finite index in $H$ and is a central extension of $\Lambda$.
Since $\Lambda \in \sd_n$, we have $H_1\in \sd_n'$.
The induction is completed.
\end{proof}

Let $\se$ denote the class of discrete groups having a finite index subgroup belonging to $\sd$.
This class $\se$ satisfies conditions (1)--(4).
In fact, conditions (1), (2) and (4) are easily checked, and condition (3) follows from Lemma \ref{lem-d}.
The inclusion $\scc \subset \se$ therefore holds.
We have shown the following:

\begin{lem}\label{lem-c-d}
Any group in $\scc$ has a finite index subgroup belonging to $\sd$.
\end{lem}

By the inclusion $\sd \subset \scc$ and condition (3), we have $\se \subset \scc$ and therefore $\scc =\se$.
We do not however use this equality in the sequel.

For any group in $\scc$, we construct its certain p.m.p.\ action and show that the associated groupoid has the HAP.
We prepare the following terminology:

\begin{defn}
Let $\calg$ be a discrete measured groupoid on a standard probability space $(X, \mu)$ and $\cals$ a subgroupoid of $\calg$.
\begin{enumerate}
\item[(i)] We say that $\cals$ is {\it co-abelian} in $\calg$ if there exist an abelian discrete group $A$ and a homomorphism $\alpha \colon \calg \to A$ with $\ker \alpha =\cals$. 
\item[(ii)] We say that $\cals$ is {\it weakly co-abelian} in $\calg$ if there exists an increasing sequence of subgroupoids of $\calg$ of finite length, $\cals =\cals_0<\cals_1<\cdots <\cals_L=\calg$, such that for any $l=0, 1,\ldots, L-1$, the subgroupoid $\cals_l$ is co-abelian in $\cals_{l+1}$.
\end{enumerate}
\end{defn}

\begin{thm}\label{thm-action}
Let $G$ be a discrete group in $\scc$.
Then there exists an ergodic p.m.p.\ action $G\c (X, \mu)$ satisfying the following:
We set $\calg =G\ltimes (X, \mu)$, and define $\rho \colon \calg \to G$ as the projection.
There exist
\begin{itemize}
\item a measurable subset $Y$ of $X$ with $\calg Y=X$;
\item a weakly co-abelian subgroupoid $\calr$ of $\calg|_Y$;
\item a treeable equivalence relation $\calr_i$ on $(X_i, \mu_i)$ indexed by $i=1,\ldots, n$;
\item an isomorphism $f\colon \calr_1\times \cdots \times \calr_n\to \calr$; and
\item a homomorphism $\rho_i\colon \calr_i \to G$ indexed by $i=1,\ldots, n$
\end{itemize}
such that for any $i=1,\ldots, n$, the following diagram commutes:
\[\xymatrix{
\ar[d]_{p_i} \cali_1\times \cdots \times \cali_{i-1}\times \calr_i\times \cali_{i+1}\times \cdots \times \cali_n \ar[r]^{\hspace{7.8em}f} & \calr \ar[r]^{\rho} & G\\
\calr_i\ar[urr]_{\rho_i} & &\\  
}\]
where $\cali_j$ is the trivial subrelation of $\calr_j$ for $j=1,\ldots, n$, and $p_i$ is the projection.
\end{thm}

\begin{proof}
Let $\scb$ denote the class of discrete groups $G$ having an ergodic p.m.p.\ action $G\c (X, \mu)$ satisfying the condition in the theorem.
To prove $\scc \subset \scb$, it is enough to show the following assertions (i)--(iii):
\begin{enumerate}
\item[(i)] Any treeable group belongs to $\scb$.
\item[(ii)] The direct product of two groups in $\scb$ belongs to $\scb$.
\item[(iii)] Any central extension of a group in $\scb$ belongs to $\scb$.
\item[(iv)] Any group having a finite index subgroup in $\scb$ belongs to $\scb$.
\end{enumerate}
In fact, assertions (i)--(iii) imply $\sd \subset \scb$, and assertion (iv) and Lemma \ref{lem-c-d} imply $\scc \subset \scb$.

Assertions (i) and (ii) hold by the definition of $\scb$.

We prove assertion (iv).
Let $G$ be a discrete group and $H$ a finite index subgroup of $G$ with $H\in \scb$.
There exists an ergodic p.m.p.\ action $H\c (X, \mu)$ satisfying the condition in the theorem.
We can check that the action of $G$ induced from the action $H\c (X, \mu)$ satisfies the condition in the theorem.
Assertion (iv) follows.

We prove assertion (iii).
Let $1\to C\to G\to \Gamma \to 1$ be an exact sequence of discrete groups such that $C$ is central in $G$ and $\Gamma \in \scb$.
We have to show $G\in \scb$.
Since $\Gamma$ belongs to $\scb$, there exists an ergodic p.m.p.\ action $\Gamma \c (X, \mu)$ satisfying the following:
We set $\calq =\Gamma \ltimes (X, \mu)$, and define $\rho \colon \calq \to \Gamma$ as the projection.
There exist
\begin{itemize}
\item a measurable subset $Y$ of $X$ with $\calg Y=X$;
\item a weakly co-abelian subgroupoid $\calr$ of $\calq|_Y$;
\item a treeable equivalence relation $\calr_i$ on $(X_i, \mu_i)$ indexed by $i=1,\ldots, n$;
\item an isomorphism $f\colon \calr_1\times \cdots \times \calr_n\to \calr$; and
\item a homomorphism $\rho_i\colon \calr_i \to \Gamma$ indexed by $i=1,\ldots, n$
\end{itemize}
such that for any $i=1,\ldots, n$, the following diagram commutes:
\begin{equation}\label{diam}
\xymatrix{
\ar[d]_{p_i} \cali_1\times \cdots \times \cali_{i-1}\times \calr_i\times \cali_{i+1}\times \cdots \times \cali_n \ar[r]^{\hspace{7.8em}f} & \calr \ar[r]^{\rho} & \Gamma\\
\calr_i\ar[urr]_{\rho_i} & &\\  
}
\end{equation}
where $\cali_j$ is the trivial subrelation of $\calr_j$ for $j=1,\ldots, n$, and $p_i$ is the projection.
Let $G$ act on $(X, \mu)$ through the quotient map from $G$ onto $\Gamma$.
We set $\calg =G\ltimes (X, \mu)$, and denote by $\theta \colon \calg \to \calq$ the quotient map.
We get the exact sequence of discrete measured groupoids,
\[1\to \cal{C}\to \calg \stackrel{\theta}{\to} \calq \to 1,\]
where $\cal{C}$ is the subgroupoid of $\calg$ associated with the trivial action of $C$ on $(X, \mu)$.
Let $\tau \colon \calg \to G$ be the projection.

\medskip

\noindent {\bf An outline of the rest of the proof.}
We will show that $\calg$ satisfies the desired condition in the theorem.
For $i=1,\ldots, n$, we set
\[\overline{\calr}_i=\cali_1\times \cdots \times \cali_{i-1}\times \calr_i\times \cali_{i+1}\times \cdots \times \cali_n.\]
The subgroupoid $f(\overline{\calr}_i)$ of $\calq |_Y$ is liftable to a subgroupoid of $\calg|_Y$ because $\calr_i$ is treeable.
The subgroupoid $\calr =f(\calr_1\times \cdots \times \calr_n)$ is however not necessarily liftable.
We will show that there exists a co-abelian subrelation $\cals_i$ of $\calr_i$ for $i=1,\ldots, n$ such that $f(\cals_1\times \cdots \times \cals_n)$ is liftable to a subgroupoid of $\calg|_Y$, denoted by $\calh$.
Let $\cal{K}$ denote the subgroupoid of $\calg|_Y$ generated by $\cal{C}|_Y$ and $\calh$.
It follows that $\calh$ is co-abelian in $\cal{K}$, that $\cal{K}$ is co-abelian in $\theta^{-1}(\calr)$, that $\theta^{-1}(\calr)$ is weakly co-abelian in $\calg|_Y$ by assumption, and therefore that $\calh$ is weakly co-abelian in $\calg|_Y$.
We also have that $\calh$ is isomorphic to $\cals_1\times \cdots \times \cals_n$.
The relation $\cals_i$ is a subrelation of the treeable relation $\calr_i$, and is thus treeable thanks to Gaboriau \cite[Th\'eor\`eme IV.4]{gab-cost}.

\medskip

We now follow this outline.
Pick a section $u\colon \Gamma \to G$ of the quotient map from $G$ onto $\Gamma$.
Let $\sigma \colon \Gamma \times \Gamma \to C$ be the 2-cocycle associated with $u$.
It is defined by
\[\sigma(\gamma, \delta)u(\gamma \delta)=u(\gamma)u(\delta)\quad \textrm{for}\ \gamma, \delta \in \Gamma.\]
For $i=1,\ldots, n$, we define the 2-cocycle $\sigma_i\colon \calr_i^{(2)} \to C$ by
\[\sigma_i(g, h)=\sigma(\rho_i(g), \rho_i(h))\quad  \textrm{for}\ (g, h)\in \calr_i^{(2)}.\]
By Corollary \ref{cor-treeable}, there exists a measurable map $\varphi_i\colon \calr_i\to C$ such that
\[\sigma_i(g, h)=\varphi_i(gh)\varphi_i(g)^{-1}\varphi_i(h)^{-1}\quad \textrm{for\ a.e.}\ (g, h)\in \calr_i^{(2)}.\]
We then have the homomorphism $\tau_i \colon \calr_i\to G$ defined by
\[\tau_i(g)=\varphi_i(g)u(\rho_i(g))\quad \textrm{for}\ g\in \calr_i.\]
For $j=1,\ldots, n$ and $x_j\in X_j$, we set $e_{x_j}=(x_j, x_j)\in \cali_j$.
Let $f_0\colon X_1\times \cdots \times X_n\to Y$ be the isomorphism induced by $f$.
We define a map $F_i\colon \overline{\calr}_i\to \calg|_Y$ by
\begin{align}\label{fi}
F_i(e_{x_1},\ldots, e_{x_{i-1}}, g, e_{x_{i+1}},\ldots, e_{x_n})=(\tau_i(g), f_0(x_1,\ldots, x_n))
\end{align}
for $x_j\in X_j$ for $j=1,\ldots, n$ and $g\in \calr_i$ with its source $x_i$.
By commutativity of diagram (\ref{diam}), the range of the right hand side in equation (\ref{fi}) is $f_0(x_1,\ldots, x_{i-1}, y_i, x_{i+1},\ldots, x_n)$, where $y_i$ is the range of $g$.
The map $F_i$ is thus a homomorphism.
Let $\calg_i$ denote the image of $F_i$.
We have $\theta \circ F_i=f$ on $\overline{\calr}_i$, and also have the following commutative diagram:
\begin{equation}\label{comm-fi}
\xymatrix{
\ar[d]_{p_i} \overline{\calr}_i \ar[r]^{F_i} & \calg_i \ar[r]^{\tau} & G\\
\calr_i\ar[urr]_{\tau_i} & &\\  
}
\end{equation}
The groupoid $\calg_i$ is a lift of $f(\overline{\calr}_i)$, i.e., the restriction of $\theta$ to $\calg_i$ is an isomorphism onto $f(\overline{\calr}_i)$.
The map $\theta$ is however not necessarily injective on the subgroupoid of $\calg|_Y$ generated by $\calg_1,\ldots, \calg_n$.
We will find a co-abelian subgroupoid $\calh_i$ of $\calg_i$ for $i=1,\ldots, n$ such that $\theta$ is injective on the subgroupoid of $\calg|_Y$ generated by them.

Fix integers $i$, $j$ with $1\leq i<j\leq n$.
We define a homomorphism $\omega^{ij}\colon \calr_i\to H^1(\calr_j, C)$ by
\[\omega^{ij}_{g_i}(g_j)=\tau_j(g_j)^{-1}\tau_i(g_i)^{-1}\tau_j(g_j)\tau_i(g_i)\quad \textrm{for}\ g_i\in \calr_i\ \textrm{and}\ g_j\in \calr_j.\]
The element $\omega^{ij}_{g_i}(g_j)$ in fact belongs to $C$ because the product of elements of $\calg$,
\begin{align}\label{comm}
(\tau_j(g_j)^{-1}, \ast)(\tau_i(g_i)^{-1}, \ast)(\tau_j(g_j), \ast)(\tau_i(g_i), x),
\end{align}
belongs to $\cal{C}=\ker \theta$, where $x$ is any element of $Y$ such that the coordinates of $f_0^{-1}(x)$ in $X_i$ and $X_j$ are the sources of $g_i$ and $g_j$, respectively, and the elements of $Y$ put in ``$\ast$" is uniquely determined so that the product is well defined.

For a.e.\ $g_i\in \calr_i$, the map $\omega^{ij}_{g_i}\colon \calr_j\to C$ is a homomorphism.
In fact, for a.e.\ $(g_j, h_j)\in \calr_j^{(2)}$, we have
 \begin{align*}
 \omega^{ij}_{g_i}(g_jh_j) & =\tau_j(g_jh_j)^{-1}\tau_i(g_i)^{-1}\tau_j(g_jh_j)\tau_i(g_i)\\
& =\tau_j(h_j)^{-1}(\tau_j(g_j)^{-1}\tau_i(g_i)^{-1}\tau_j(g_j)\tau_i(g_i))\tau_i(g_i)^{-1}\tau_j(h_j)\tau_i(g_i)\\
& = \omega^{ij}_{g_i}(g_j)\omega^{ij}_{g_i}(h_j).
\end{align*}
The map $\omega^{ij}\colon \calr_i\to H^1(\calr_j, C)$ is a homomorphism.
In fact, for a.e.\ $(g_i, h_i)\in \calr_i^{(2)}$ and a.e.\ $g_j\in \calr_j$, we have
\begin{align*}
\omega^{ij}_{g_ih_i}(g_j) & =\tau_j(g_j)^{-1}\tau_i(g_ih_i)^{-1}\tau_j(g_j)\tau_i(g_ih_i)\\
& = \tau_j(g_j)^{-1}\tau_i(h_i)^{-1}\tau_j(g_j)(\tau_j(g_j)^{-1}\tau_i(g_i)^{-1}\tau_j(g_j)\tau_i(g_i))\tau_i(h_i)\\
& = \omega^{ij}_{g_i}(g_j)\omega^{ij}_{h_i}(g_j).
\end{align*}

We obtained a homomorphism $\omega^{ij}\colon \calr_i\to H^1(\calr_j, C)$.
The image of this homomorphism depends only on $\tau_i(g_i)\in G$.
More precisely, for a.e.\ $g_i, h_i\in \calr_i$ with $\tau_i(g_i)=\tau_i(h_i)$, we have $\omega^{ij}_{g_i}=\omega^{ij}_{h_i}$.
It follows that the image $\omega^{ij}(\calr_i)$ is countable.

We are ready to show that the action $G\c (X, \mu)$ is desirable, i.e., there exist a weakly co-abelian subgroupoid of $\calg|_Y$, a treeable equivalence relation, etc.\ in the theorem.
We set $\cals_n=\calr_n$.
For $i=1,\ldots, n-1$, we set $\cals_i=\bigcap_{j=i+1}^n\ker \omega^{ij}$.
The subrelation $\cals_i$ is co-abelian in $\calr_i$ because $\omega^{ij}(\calr_i)$ is countable and abelian.
By \cite[Th\'eor\`eme IV.4]{gab-cost}, the relation $\cals_i$ is treeable.
For $i=1,\ldots, n$, we set
\[\calh_i=F_i(\cali_1\times \cdots \times \cali_{i-1}\times \cals_i\times \cali_{i+1}\times \cdots \times \cali_n).\]
Let $\calh$ be the subgroupoid of $\calg|_Y$ generated by $\calh_1,\ldots, \calh_n$.
We have the homomorphism $F\colon \cals_1\times \cdots \times \cals_n\to \calh$ that extends the isomorphisms $F_1,\ldots, F_n$ because the product in (\ref{comm}) is the unit of $\calg$ at $x$ for a.e.\ $g_i\in \cals_i$ and $g_j\in \cals_j$.
The map $F$ is surjective by the definition of $\calh$, and is injective because $\theta \circ F=f$ on $\cals_1\times \cdots \times \cals_n$.
It follows that $F$ is an isomorphism.
For any $i=1,\ldots, n$, commutativity of diagram (\ref{comm-fi}) implies commutativity of the following diagram:
\[\xymatrix{
\ar[d]_{p_i} \cali_1\times \cdots \times \cali_{i-1}\times \cals_i\times \cali_{i+1}\times \cdots \times \cali_n \ar[r]^{\hspace{7.8em} F} & \calh \ar[r]^{\tau} & G\\
\cals_i\ar[urr]_{\tau_i} & &\\  
}\]

The remaining task is to prove that $\calh$ is weakly co-abelian in $\calg|_Y$.
We denote by $\cal{K}$ the subgroupoid of $\calg|_Y$ generated by $\calh$ and $\cal{C}|_Y$.
We have the isomorphism from $C\times \calh$ onto $\cal{K}$ sending $(c, h)$ to $ch$ for $c\in C$ and $h\in \calh$.
The subgroupoid $\calh$ is therefore co-abelian in $\cal{K}$ with $C$ the quotient.
The subgroupoid $\cal{K}$ is co-abelian in $\theta^{-1}(\calr)$ because $\cals_1\times \cdots \times \cals_n$ is co-abelian in $\calr_1\times \cdots \times \calr_n$.
The subgroupoid $\theta^{-1}(\calr)$ is weakly co-abelian in $\calg|_Y$ because $\calr$ is weakly co-abelian in $\calq|_Y$.
It follows that $\calh$ is weakly co-abelian in $\calg|_Y$.

We proved that $G$ belongs to $\scb$.
Assertion (iii) follows.
\end{proof}

\begin{proof}[Proof of Theorem \ref{thm-h}]
Let $G$ be a group in $\scc$.
We show that $G$ has the HAP.
By Theorem \ref{thm-action}, there exists an ergodic p.m.p.\ action $G\c (X, \mu)$ such that, setting $\calg =G\ltimes (X, \mu)$, we have a measurable subset $Y$ of $X$ with $\calg Y=X$ and a weakly co-abelian subgroupoid $\calr$ of $\calg|_Y$ isomorphic to the direct product of finitely many treeable equivalence relations.
By \cite[Theorem 3]{ana-h} or \cite[Proposition 6]{ueda}, any treeable equivalence relation has the HAP.
It follows that $\calr$ has the HAP, and so do $\calg|_Y$ and $\calg$ by Propositions \ref{prop-hap-ext} and \ref{prop-hap-rest}.
The group $G$ therefore has the HAP by Proposition \ref{prop-group-oid}.

We next suppose that $G$ is a discrete group and $C$ is a central subgroup of $G$ such that $G/C$ belongs to $\scc$.
We show that the pair $(G, C)$ has the gHAP.
In the end of the proof of Theorem \ref{thm-action}, letting $\Gamma$ be $G/C$, we obtained an ergodic p.m.p.\ action $G\c (X, \mu)$ such that, setting $\calg =G\ltimes (X, \mu)$ and $\cal{C}=C\ltimes (X, \mu)$, we have the following three conditions:
\begin{itemize}
\item The group $C$ acts on $(X, \mu)$ trivially.
\item There exist a measurable subset $Y$ of $X$ with $\calg Y=X$, a weakly co-abelian subgroupoid $\cal{K}$ of $\calg|_Y$, and a subgroupoid $\calh$ of $\cal{K}$ such that $\cal{K}$ is generated by $\cal{C}|_Y$ and $\calh$; and the map from $C\times \calh$ into $\cal{K}$ sending $(c, h)$ to $ch$ for $c\in C$ and $h\in \calh$ is an isomorphism.
\item The groupoid $\calh$ is isomorphic to the direct product of finitely many treeable equivalence relations, and therefore has the HAP.
\end{itemize}
Let $\iota \colon \cal{C}\to \calg$ denote the inclusion.
These three conditions imply that the pair $(\cal{K}, (C, \iota))$ has the gHAP.
Applying Propositions \ref{prop-ghap-ext}, \ref{prop-ghap-rest} and \ref{prop-group-oid} in this order, we see that the pairs $(\calg|_Y, (C, \iota))$, $(\calg, (C, \iota))$ and $(G, C)$ have the gHAP.
\end{proof}


\appendix

\section{Comparison with the HAPs of Jolissaint and Anantharaman-Delaroche}\label{sec-app}

The HAP of a p.m.p.\ discrete measured equivalence relation was introduced by Jolissaint \cite{jol-h}, who proved that it has his HAP if and only if the associated von Neumann algebra has property H relative to the Cartan subalgebra, in the sense of Popa \cite{popa}.
The HAP of a discrete measured groupoid was introduced by Anantharaman-Delaroche \cite{ana-h}, who proved that it has her HAP if and only if the associated von Neumann algebra has the Haagerup property relative to the Cartan subalgebra, in the sense of Boca \cite{boca}.
In this appendix, we compare their HAPs and our HAP in Definition \ref{defn-hap}.
We refer to \cite[Proposition 2.2]{popa} for comparison between the above properties of Popa and Boca.
We also refer to \cite{cs}, \cite{cost}, \cite{jol-h-vn} and \cite{ot} and references therein for historical background on the HAP of von Neumann algebras and recent approaches to introduce the HAP for arbitrary von Neumann algebras.

Jolissaint's HAP (\cite[Definition 1.2]{jol-h}) can naturally be extended for a discrete measured groupoid as follows:

\begin{defn}\label{defn-hap-jol}
Let $\calg$ be a discrete measured groupoid on a standard probability space $(X, \mu)$.
We say that $\calg$ has the {\it Jolissaint Haagerup-property (JHAP)} if there exists a sequence $(\phi_n)_n$ of positive definite functions on $\calg$ such that
\begin{itemize}
\item for any $n$ and a.e.\ $x\in X$, we have $\phi_n(e_x)\leq 1$, where $e_x$ is the unit of $\calg$ at $x$;
\item for any $n$ and any $\varepsilon >0$, we have $\tilde{\mu}(\{ \, g\in \calg \mid |\phi_n(g)|>\varepsilon \, \})<\infty$; and
\item for a.e.\ $g\in \calg$, we have $\phi_n(g)\to 1$ as $n\to \infty$.
\end{itemize}
\end{defn}

The following is due to Anantharaman-Delaroche (\cite[Definition 8]{ana-h}):

\begin{defn}\label{defn-hap-ad}
Let $\calg$ be a discrete measured groupoid on a standard probability space $(X, \mu)$.
We say that $\calg$ has the {\it Anantharaman-Delaroche Haagerup-property (ADHAP)} if there exists a sequence $(\phi_n)_n$ of positive definite functions on $\calg$ such that
\begin{itemize}
\item for any $n$ and a.e.\ $x\in X$, we have $\phi_n(e_x)=1$;
\item for any $n$ and any $\varepsilon >0$, we have $\tilde{\mu}(\{ \, g\in \calg \mid |\phi_n(g)|>\varepsilon \, \})<\infty$; and
\item for a.e.\ $g\in \calg$, we have $\phi_n(g)\to 1$ as $n\to \infty$.
\end{itemize}
\end{defn}

The difference between the JHAP and the ADHAP is only the normalized condition for positive definite functions.
The difference between the ADHAP and our HAP in Definition \ref{defn-hap} is only the $c_0$-property of positive definite functions.
It is obvious from definition that the ADHAP implies the JHAP, and that the ADHAP implies our HAP.

\begin{rem}
It does not seem obvious from the above definitions that the JHAP and the ADHAP depend only on the class of the measure $\mu$.
Through his characterization of the JHAP in terms of von Neumann algebras, Jolissaint observed that the JHAP for p.m.p.\ discrete measured equivalence relations depends only on the class of $\mu$ (\cite[Corollary 2.7]{jol-h}).
Similarly, through her characterization of the ADHAP in terms of von Neumann algebras, Anantharaman-Delaroche observed that the ADHAP depends only on the class of $\mu$ (see the paragraph right after \cite[Definition 8]{ana-h}).

In contrast, it follows from definition that our HAP in Definition \ref{defn-hap} depends only on the class of $\mu$.
Applying Corollary \ref{cor-hap} below, we can also deduce that the JHAP and the ADHAP for discrete measured groupoids depends only on the class of $\mu$.
\end{rem}

\begin{lem}\label{lem-ad-hap}
Let $\calg$ be a discrete measured groupoid on a standard probability space $(X, \mu)$.
Then $\calg$ has the ADHAP if and only if $\calg$ has the HAP in Definition \ref{defn-hap}.\end{lem}

\begin{proof}
We have already mentioned that the ``only if" part follows by definition.
We prove the ``if" part.
Suppose that $\calg$ has the HAP in Definition \ref{defn-hap}.
There exists a sequence $(\phi_n)_{n\in \N}$ of normalized positive definite $c_0$-functions on $\calg$ such that for a.e.\ $g\in \calg$, we have $\phi_n(g)\to 1$ as $n\to \infty$.

Fix $n\in \N$.
There exists a measurable subset $X_n$ of $X$ such that $\mu(X\setminus X_n)<1/n^2$; and for any $\varepsilon >0$, there exists a number $M_{n, \varepsilon}\in \N$ such that
\[|\{ \, g\in (\calg|_{X_n})_x\mid |\phi_n(g)|>\varepsilon \, \} |\leq M_{n, \varepsilon}\quad \textrm{for\ a.e.}\ x\in X_n.\]
We define a function $\psi_n\colon \calg \to \C$ by $\psi_n(g)=\phi_n(g)$ if $g\in \calg|_{X_n}$; $\psi_n(g)=1$ if $g=e_x$ for some $x\in X\setminus X_n$; and $\psi_n(g)=0$ otherwise.
The function $\psi_n$ is positive definite on $\calg$ by Lemma \ref{lem-rest-pdf}, and is normalized.
For a.e.\ $g\in \calg$, we have $\psi_n(g)\to 1$ as $n\to \infty$ because $\psi_n=\phi_n$ on $\calg|_{X_n}$ and $\sum_{n=1}^\infty \mu(X\setminus X_n)<\infty$.

We fix $n\in \N$ again.
Pick a number $\varepsilon$ with $0<\varepsilon <1$.
We have
\[\tilde{\mu}(\{ \, g\in \calg \mid |\psi_n(g)|>\varepsilon \, \} )=\mu(X\setminus X_n)+\tilde{\mu}(\{ \, g\in \calg|_{X_n}\mid |\psi_n(g)|>\varepsilon \, \} )\]
and
\[\tilde{\mu}(\{ \, g\in \calg|_{X_n} \mid |\psi_n(g)|>\varepsilon \, \} ) =\int_{X_n}|\{ \, g\in (\calg|_{X_n})_x \mid |\psi_n(g)|>\varepsilon \, \}|\, d\mu(x) \leq M_{n, \varepsilon}\mu(X_n).\]
We thus have $\tilde{\mu}(\{ \, g\in \calg \mid |\psi_n(g)|>\varepsilon \, \} )<\infty$.

We proved that the sequence $(\psi_n)_n$ satisfies the condition in Definition \ref{defn-hap-ad}.
It follows that $\calg$ has the ADHAP.
\end{proof}

\begin{lem}\label{lem-jol-hap}
Let $\calg$ be a discrete measured groupoid on a standard probability space $(X, \mu)$.
Then $\calg$ has the JHAP if and only if $\calg$ has the HAP in Definition \ref{defn-hap}.
\end{lem}

\begin{proof}
Suppose that $\calg$ has the JHAP.
There exists a sequence $(\phi_n)_{n\in \N}$ of positive definite functions on $\calg$ satisfying the condition in Definition \ref{defn-hap-jol}.
Fix $n\in \N$.
We set
\[X_n=\{ \, x\in X\mid \phi_n(e_x)>0\, \}\]
and define a function $\psi_n\colon \calg \to \C$ by
\[\psi_n(g)=\begin{cases}
\displaystyle \frac{\phi_n(g)}{\, \phi_n(e_{r(g)})^{1/2}\phi_n(e_{s(g)})^{1/2}\, } & \textrm{if}\ g\in \calg|_{X_n}\\
1 & \textrm{if}\ g=e_x \ \textrm{for\ some}\ x\in X\setminus X_n\\
0 & \textrm{otherwise}.
\end{cases}
\]
The function $\psi_n$ is positive definite by Lemma \ref{lem-rest-pdf}, and is normalized by definition.

We claim that $\psi_n$ is $c_0$.
To prove it, pick $\delta >0$ and $\varepsilon >0$.
There exist a measurable subset $Y$ of $X_n$ and $c>0$ such that $\mu(X_n\setminus Y)<\delta$ and for a.e.\ $x\in Y$, we have $\phi_n(e_x)>c$.
We set $Z=Y\cup (X\setminus X_n)$.
We have $\mu(X\setminus Z)<\delta$.
For a.e.\ $g\in \calg|_Z$ with $|\psi_n(g)|>\varepsilon$, we have either $g=e_x$ for some $x\in X\setminus X_n$ or $g\in \calg|_Y$ and
\[|\phi_n(g)|=\phi_n(e_{r(g)})^{1/2}\phi_n(e_{s(g)})^{1/2}|\psi_n(g)|>c\varepsilon.\]
We thus have
\[\tilde{\mu}(\{ \, g\in \calg|_Z\mid |\psi_n(g)|>\varepsilon \, \})\leq \mu(X\setminus X_n)+\tilde{\mu}(\{ \, g\in \calg|_Y\mid |\phi_n(g)|>c\varepsilon \, \})<\infty.\]
The claim follows.

For a.e.\ $g\in \calg$, we have $\psi_n(g)\to 1$ as $n\to \infty$ by the definition of $\psi_n$.
It follows that $\calg$ has the HAP in Definition \ref{defn-hap}.

We proved that the ``only if" part of the lemma.
The ``if" part follows because the HAP in Definition \ref{defn-hap} implies the ADHAP by Lemma \ref{lem-ad-hap}, and the ADHAP implies the JHAP by definition.
\end{proof}

Combining Lemmas \ref{lem-ad-hap} and \ref{lem-jol-hap}, we obtain the following:

\begin{cor}\label{cor-hap}
For any discrete measured groupoid, all of the JHAP, the ADHAP and the HAP in Definition \ref{defn-hap} are equivalent.
\end{cor}




\end{document}